\documentclass[draft,preprint,12pt,numbers,sort&compress]{elsarticle}
\usepackage{mathrsfs}
\usepackage{amssymb}
\usepackage{amsthm}
\usepackage{mathrsfs}
\usepackage[centertags]{amsmath}
\usepackage{amsfonts}

\usepackage{color}

\input amssym.def
\input amssym.tex

\date{}

\newtheorem{Theorem}{Theorem}[section]

\newtheorem{Lemma}{Lemma}[section]

\newcommand\R{\mbox{\bf R}}
\newcommand\N{\mbox{\bf N}}
\newcommand\SR{\mbox{\scriptsize\bf R}}

\newcommand{\definition}{{\lower .5ex
  \hbox{$\>\>\stackrel{\triangle}{=}\>\>$} }}
\newcommand\supp{\mathop{\rm supp}}

\begin{document}

\baselineskip=22pt
\thispagestyle{empty}

\mbox{}
\bigskip

\begin{center}{\Large\bf The Cauchy problem for the shallow water type}\\[1ex]
{\Large \bf  equations in low regularity spaces on the circle }

{Wei YAN$^a$,\quad  Yongsheng LI$^b$, \quad
\quad Xiaoping Zhai$^c$ and Yimin Zhang $^d$}\\[2ex]

{$^a$School of Mathematics and Information Science, Henan Normal University,}\\
{Xinxiang, Henan 453007, P. R. China}\\[1ex]

{$^b$Department of Mathematics, South China University of Technology,}\\
{Guangzhou, Guangdong 510640, P. R. China}\\[1ex]

{$^c$School of Mathematics Computational Science, Sun Yat-sen University}\\[1ex]
{Guangzhou, Guangdong 510275, P. R. China}\\[1ex]
{$^d$Wuhan Institute of Physics and Mathematics, Chinese Academy of Sciences, Wuhan, Hubei 430071, P. R. China}

\end{center}

\bigskip
\bigskip

\noindent{\bf Abstract.}  In this paper, we investigate the Cauchy problem
for the  shallow water type equation
\begin{eqnarray*}
      u_{t}+\partial_{x}^{3}u
     + \frac{1}{2}\partial_{x}(u^{2})+\partial_{x}
     (1-\partial_{x}^{2})^{-1}\left[u^{2}+\frac{1}{2}u_{x}^{2}\right]=0,x\in {\mathbf T}=\R/2\pi
     \lambda
\end{eqnarray*}
 with low regularity data in the periodic settings and $\lambda\geq1$. We prove that the bilinear
 estimate in  $X_{s,b}$ with $s<\frac{1}{2}$ is invalid.
    We also  prove that the
 problem
 is locally well-posed in $H^{s}(\mathbf{T})$ with $\frac{1}{6}<s<\frac{1}{2}$ for small initial data.
 The result of  this paper improves the result of case
 $j=1$ of  Himonas and  Misiolek (Communications in Partial Differential Equations,
23(1998), 123-139.).
 The new ingredients   are
 some new function spaces and some new Strichartz estimates.

\bigskip


\bigskip
\noindent {\bf Short Title:} Cauchy problem for shallow water type equation

\bigskip
\noindent {\bf Corresponding Author:} W. YAN

\bigskip
\noindent {\bf Email Address:}yanwei19821115@sina.cn

\bigskip
\noindent{\bf Fax number:} +86-0373-3326174

\bigskip

\leftskip 0 true cm \rightskip 0 true cm

\newpage{}

\begin{center}{\Large\bf The Cauchy problem for the shallow water type}\\[1ex]
{\Large \bf  equations in low regularity spaces on the circle}

{Wei YAN$^a$,\quad  Yongsheng LI$^b$, \quad
\quad Xiaoping Zhai$^c$ and Yimin Zhang $^d$}\\[2ex]

{$^a$School of Mathematics and Information Science, Henan Normal University,}\\
{Xinxiang, Henan 453007, P. R. China}\\[1ex]

{$^b$Department of Mathematics, South China University of Technology,}\\
{Guangzhou, Guangdong 510640, P. R. China}\\[1ex]

{$^c$School of Mathematics Computational Science, Sun Yat-sen University}\\[1ex]
{Guangzhou, Guangdong 510275, P. R. China}\\[1ex]
{$^d$Wuhan Institute of Physics and Mathematics, Chinese Academy of Sciences, Wuhan, Hubei 430071, P. R. China}

\end{center}

\noindent{\bf Abstract.}   In this paper, we investigate the Cauchy problem
for the  shallow water type equation
\begin{eqnarray*}
      u_{t}+\partial_{x}^{3}u
     + \frac{1}{2}\partial_{x}(u^{2})+\partial_{x}
     (1-\partial_{x}^{2})^{-1}\left[u^{2}+\frac{1}{2}u_{x}^{2}\right]=0,x\in {\mathbf T}=\R/2\pi
     \lambda
\end{eqnarray*}
 with low regularity data in the periodic settings and $\lambda\geq1$. We prove that the bilinear
 estimate in  $X_{s,b}$ with $s<\frac{1}{2}$ is invalid.
    We also  prove that the
 problem
 is locally well-posed in $H^{s}(\mathbf{T})$ with $\frac{1}{6}<s<\frac{1}{2}$ for small initial data.
 The result of  this paper improves the result of case
 $j=1$ of  Himonas and  Misiolek (Communications in Partial Differential Equations,
23(1998), 123-139.).
 The new ingredients   are
 some new function spaces and some new Strichartz estimates.
\bigskip

{\large\bf 1. Introduction}
\bigskip

\setcounter{Theorem}{0} \setcounter{Lemma}{0}

\setcounter{section}{1}

In this paper, we consider the Cauchy problem for the shallow water type equation
\begin{eqnarray}
&& u_{t}+\partial_{x}^{3}u
     + \frac{1}{2}\partial_{x}(u^{2})+\partial_{x}(1-\partial_{x}^{2})^{-1}
     \left[u^{2}+\frac{1}{2}u_{x}^{2}\right]=0,\label{1.01}\\
    &&u(x,0)=u_{0}(x),\quad x\in \mathbf{T}=\R/2\pi\lambda,\lambda\geq1. \label{1.02}
\end{eqnarray}
Obviously, (\ref{1.01}) is the higher order modification of
the Camassa-Holm equation
\begin{eqnarray}
u_{t}+ \frac{1}{2}\partial_{x}(u^{2})+\partial_{x}(1-\partial_{x}^{2})^{-1}
\left[u^{2}+\frac{1}{2}u_{x}^{2}\right]=0\label{1.03}
 \end{eqnarray}
with nonlocal form. Equation (\ref{1.03})
 was derived by Camassa and Holm as a nonlinear model for water wave motion
in shallow channels with  the aid of  an asymptotic expansion directly in the Hamiltonian
for Euler equations \cite{CH,FF}. Camassa-Holm equation has been studied extensively in the last three decades, for instance, see
\cite{BCARMA, BCAA, CH,C, C2000,C2001,C2006,CE, CECPAM,CE1998,CEA,CEB2007,
CE2000,CJ,  CM,CMCPAM, CMP, CS, CSNS, CSPLA, CK,  CKL, CLa,
DH,D2003,DGH,EY,EYJFA,FF,HMPZ,HM2005,Iv2006,Iv2007,J,KL,Lak,L2004IMRN,
LIMRN, L2005, LPA,L2005JDE,  Ko,R, XZCPAM, XZCPDE,Y}.
 Fokas and Fuchssteiner \cite{FF} proved that the
Camassa-Holm equation possesses bi-Hamiltonian structure.
Camassa and  Holm \cite{CH}, proved that
the Camassa-Holm equation is completely integrable.
Camassa et.al. \cite{CHH} proved that the Camassa-Holm equation
possesses peaked solitary waves which are orbitally stable and
interact like solitons \cite{BSS,CS}.
Constantin and his co-authors \cite{C2006, CEB2007} proved that
the peaked solitary waves replicate a
characteristic for the waves of great height-waves of largest
amplitude which are exact solutions of the governing equations
for water waves, see also \cite{To}.
 The result of \cite{R} implies that the Cauchy problem
for the Camassa-Holm equation is locally well-posed in $H^{s}(\R)$ with $s>\frac{3}{2}$.
Under some assumptions on the initial data, Constantin and Escher \cite{CEA, Con2000}
proved that the Cauchy problem for the Camassa-Holm equation possesses  not only  the global strong solutions and
 but also finite time blow-up solutions.
Constantin and Molinet \cite{CM}  proved that the Cauchy problem for the Camassa-Holm equation possesses
 the global weak solution in $H^{1}(\R),$ see also \cite{XZCPAM, XZCPDE}.
Escher and Yin \cite{EY, EYJFA} studied the initial-boundary value problem for the Camassa-Holm equation.
Lenells \cite{L2004IMRN} studied the correspondence between the KdV and Camassa-Holm equation and
the stability of periodic peakons  \cite{LIMRN}. In contrast to the Camassa-Holm
equation (\ref{1.03}), equation (\ref{1.01}) loses integrability in the sense that the equation (\ref{1.01})
is equivalent to  a linear flow at constant speed in the right action-angle-variables,
see the discussion in \cite{CM}, but, as shown in Theorem 1.2 of \cite{LYLH},
global existence prevails while for Camassa-Holm equation the development of blow-up in finite time is quite frequent.

Omitting the last term,  (\ref{1.01})  yields the  Korteweg-de Vries  equation
\begin{eqnarray}
u_{t}+\partial_{x}^{3}u+\frac{1}{2}\partial_{x}(u^{2})=0,\label{1.04}
\end{eqnarray}
which possesses the bi-Hamiltonian structure and completely integrable
and infinite conservation laws. In the recent period, mathematical studies have focus on the
Cauchy problem for the KdV
equation, for instance, see \cite{Bourgain93, Bourgain97, CKSTT,KPV1996,KPV2001,Kis,T}.
Especially, the Fourier restriction norm method which is introduced by Bourgain \cite{B,Bourgain93}
is an effective tool in solving the Cauchy problem for
dispersive equations in low regularity. Using the Fourier restriction norm method,
Kenig et. al. \cite{KPV1996}  proved that  the Cauchy problem for the periodic
KdV equation is
locally well-posed in $H^{s}(\mathbf{T})$ with $s\geq-\frac{1}{2}$.
Bourgain \cite{Bourgain97}
proved that the Cauchy problem for the periodic KdV  equation is ill-posed in
$H^{s}([0,2\pi))$
with $s< -\frac{1}{2}$  in the sense that the solution map is not $C^{3}.$
By using  the I-method which is  the modification   of  high-low frequency
technique introduced by Bourgain in \cite{B98}.  Colliander et.al. \cite{CKSTT}
proved that the Cauchy problem for the
periodic KdV  equation is globally well-posed in $H^{s}(\mathbf{T})$ with $s\geq -\frac{1}{2}.$
By using the inverse scattering method,
Kappeler and  Topalov \cite{KT2006} proved that   the  Cauchy  problem for the KdV  equation
is globally
well-posed  in $H^{s}(\mathbf{T})$   with $s\geq -1$. Recently, by using short time Bourgain spaces,
Molinet \cite{Molinet} proved that the  Cauchy  problem for the KdV  equation  is
ill-posed in
$H^{s}(\mathbf{T})$   with $s< -1$. From \cite{KPV1996,KPV2001}, we know that $s=-\frac{3}{4}$ is the critical indices for the well-posedness of the KdV equation on the real line.
By using the $I$-method and
modified Bourgain spaces, Guo \cite{G} and
Kishimoto \cite{Kis} established the global well-posedness result of the Cauchy
problem for the KdV
equation in $H^{-3/4}(\R).$
Very recently, Liu \cite{L} proved that the smooth solutions satisfy a-prior local
time $H^{s}(\R)$ bound
in terms of the $H^{s}$ size of the initial data for $s\geq -\frac{4}{5}$ on the real line.

Many people have investigate the Cauchy problem for (\ref{1.01}), for instance, see
\cite{Go,B,BDIE,B2003,HM1998,HM2000,HM,LJ,O,WC,LYY,YL,LY,LYLH} and the references therein.
Compared with the Camassa-Holm equation, the higher-order terms
arise because of the desire to  go beyond the regime of waves of small amplitude, to
capture waves of moderate amplitude, see \cite{CLa}.
 Himonas and  Misiolek \cite{HM1998} proved that (\ref{1.01})-(\ref{1.02}) are locally well-posed in $H^{s}([0,2\pi))$
 with $s\geq\frac{1}{2}$ for small initial data and are globally well-posed in $H^{1}([0,2\pi))$ for small initial data.
Himonas and  Misiolek \cite{HM} proved that (\ref{1.01})-(\ref{1.02}) are locally well-posed in $H^{s}([0,2\pi))$
 with $s\geq\frac{1}{2}$ for arbitrary initial data and are globally well-posed in $H^{1}([0,2\pi))$. A nature question one would ask is
 that: what will happen when $s<\frac{1}{2}$? As far as we know, there is no result for this case, and it is the reason why we consider this problem .

In this paper, firstly, we prove that the bilinear
 estimate in $X_{s,b}$ with $s<\frac{1}{2}$ is invalid.
 Then,  by exploiting the spirit of \cite{BT,Kato,KT,NTT,IK,IKT,TT} and  introducing some  new function
spaces and Strichartz estimates which are
used to establish the bilinear estimates and  the fixed point Theorem, we prove that
the Cauchy problem for (\ref{1.01}) is locally well-posed in $H^{s}(\mathbf{T})$ with
$\frac{1}{6}< s<\frac{1}{2}$ for small initial data,
which improves the result of case
 $j=1$ of \cite{HM1998}.

From Theorem 1.1 below, we know that the standard Fourier restriction norm method $W^{s}$ is not effective, thus we must adapt the modified
Fourier restriction norm method, in other words, motivated by Theorem 1.1, we make a suitable modification of standard Bourgain space $Z^{s}$ below.

We present some notations before stating the main results. $0<\epsilon\ll1$ means
that $0<\epsilon<\frac{1}{10^{9}}$. $C$ is a positive constant
 which may vary from line to line.  $A\sim B$ means that $|B|\leq |A|\leq 4|B|$.
 $A\gg B$ means that $|A|> 4|B|.$ $a\vee b={\rm max}\left\{a,b\right\}.$
 $a\wedge b={\rm min}\left\{a,b\right\}.$ Let $\Psi \in C_{0}^{\infty}(\R)$ be an even function such
  that $\Psi \geq 0,$ $\supp \Psi \subset [-\frac{3}{2},\frac{3}{2}]$,
 $\Psi= 1$ on $[-\frac{5}{4},\frac{5}{4}]$ and
 $\Psi_{k}=\Psi(2^{-k}\xi)-\Psi(2^{-k+1}\xi).$
Throughout this paper,
 $\dot{Z}:=Z- \{ 0\}$ and $\dot{Z}^{+}:=Z^{+}- \{ 0\}$.    Denote by
 $(dk)_{\lambda}$ the normalized counting measure on $\dot{Z_{\lambda}}=
 \frac{\dot{Z}}{\lambda}$:
 \begin{eqnarray*}
 \int a(x)(dk)_{\lambda}=\frac{1}{\lambda}\sum_{k\in  \dot{Z}_{\lambda}}a(k).
 \end{eqnarray*}
 Denote by
$
 \mathscr{F}_{x}f(k)=\int_{0}^{2\pi \lambda}e^{- i kx}f(x)dx
$
   the Fourier transformation of a function $f$
 defined on
  $\mathbf{T}$ with  respect to the space variable
 and we have the Fourier inverse transformation formula
 \begin{eqnarray*}
 f(x)=\int e^{ i kx} \mathscr{F}_{x}f(k)(dk)_{\lambda}=\frac{1}{\lambda}\sum_{k \in
 \dot{Z}_{\lambda}}e^{ i kx}\mathscr{F}_{x}f(k).
 \end{eqnarray*}
 Denote  by
$
 \mathscr{F}_{t}f(\tau)=\int_{\SR}e^{- i t\tau}f(t)dt
$
 the Fourier transformation of a function $f$
  with the respect to the time variable
  and we have the Fourier inverse transformation formula
 \begin{eqnarray*}
 f(t)=\int e^{ i t\tau} \mathscr{F}_{t}f(\tau)d\tau.
 \end{eqnarray*}
 We define
\begin{eqnarray*}
S(t)\phi(x)=\int e^{i kx }e^{it k^{3}}\mathscr{F}_{x}\phi(k)(dk)_{\lambda}.
\end{eqnarray*}
Denote by
\begin{eqnarray*}
\mathscr{F}f(k,\tau)=\int\int_{0}^{2\pi \lambda}e^{-i kx}e^{-i t\tau}f(x,t)dxdt
\end{eqnarray*}
the space-time  Fourier transform
 for $k\in \dot{Z}_{\lambda}$ and $\tau\in \R$ by
and this transformation is inverted by
\begin{eqnarray*}
f(x,t)=\int\int e^{i kx}e^{i t\tau}\mathscr{F}f(k,\tau)(dk)_{\lambda}d\tau.
\end{eqnarray*}
Let
\begin{eqnarray*}
\Lambda^{-1}=\mathscr{F}^{-1}\langle\tau-k^{3}\rangle^{-1}\mathscr{F}, \mathscr{F}_{x}J^{s}u=\langle k\rangle ^{s}\mathscr{F}_{x}u(k).
\end{eqnarray*}
It is easily checked that
\begin{eqnarray*}
&&\|f\|_{L^{2}(\mathbf{T})}=\|\mathscr{F}_{x}f\|_{L^{2}((dk)_{\lambda})},\\
&&\int_{0}^{2\pi \lambda}f(x)\overline{g(x)}dx=\int \mathscr{F}_{x}f(k)\overline{\mathscr{F}_{x}
f(k)}(dk)_{\lambda},\\
&&\mathscr{F}_{x}(fg)=\mathscr{F}_{x}f*\mathscr{F}_{x}g=\int \mathscr{F}_{x}f(k-k_{1})
\mathscr{F}_{x}g(k_{1})(dk_{1})_{\lambda}.
\end{eqnarray*}
Let
\begin{eqnarray*}
&&P(k)=-k^{3},\sigma=\tau+P(k),\quad \sigma_{j}=\tau_{j}+P(k_{j}).
\end{eqnarray*}
$\eta(t)$ is a smooth function with $\supp \eta (t)\subset [-1,2]$ and $\eta=1$ on $[-1,1]$.
We define the Sobolev space $H^{s}(\mathbf{T})$ with the norm
\begin{eqnarray*}
\|f\|_{H^{s}(\mathbf{T})}=\|\langle k\rangle^{s}\mathscr{F}_{x}f(k)\|_{L^{2}((dk)_{\lambda})}
\end{eqnarray*}
and define the $X_{s,b}$ spaces for $2\pi \lambda$-periodic KdV via the norm
\begin{eqnarray*}
\|u\|_{X_{s,b}(\mathbf{T} \times \SR)}=\left\|\langle k\rangle^{s} \left\langle \tau
+P(k)\right\rangle^{b}\mathscr{F}u(k,\tau)\right\|_{L^{2}((dk)_{\lambda}(d\tau))}
\end{eqnarray*}
and define the $Y^{s}$ space defined via the norm
\begin{eqnarray*}
\|u\|_{Y^{s}}=\left\|\langle k\rangle^{s}\mathscr{F}
u(k,\tau)\right\|_{L^{2}((dk)_{\lambda}L^{1}(d\tau))}.
\end{eqnarray*}
Let
\begin{eqnarray*}
&&D_{1}=\left\{(\tau,k)\in \R\times \dot{Z}_{\lambda}:|\sigma|
\leq 2|k|^{2},|k|\geq1\right\},\\
&&D_{2}=\left\{(\tau,k)\in \R\times \dot{Z}_{\lambda}: 2|k|^{2}<
 |\sigma|\leq 6|k|^{3},|k|\geq1\right\},\\
&&D_{3}=\left\{(\tau,k)\in \R\times \dot{Z}_{\lambda}:|\sigma|>
 6|k|^{3},|k|\geq1\right\},\\
&&D_{4}=\left\{(\tau,k)\in \R\times \dot{Z}_{\lambda}:|\sigma|>
6|k|^{3},\frac{1}{\lambda}\leq |k|\leq1\right\},\\
&&D_{5}=\left\{(\tau,k)\in \R\times \dot{Z}_{\lambda}:|\sigma|
\leq 6|k|^{3},\frac{1}{\lambda}\leq |k|\leq1\right\}.
\end{eqnarray*}
We define $Z^{s}$ space as  follows:
\begin{eqnarray*}
\|u\|_{Z^{s}}=\|P_{D_{1}\cup D_{5}}u\|_{X_{s,\frac{5}{6}-\epsilon}}
+\|P_{D_{2}}u\|_{X_{1-s,s+\frac{1}{3}-\epsilon}}+
\|P_{D_{3}\cup D_{4}}u\|_{X_{1-s,s}}+\left\|u\right\|_{Y^{s}}
\end{eqnarray*}
and
\begin{eqnarray*}
\|u\|_{W^{s}}=\|u\|_{X_{s,\frac{1}{2}}}+\left\|u\right\|_{Y^{s}}.
\end{eqnarray*}
We define $Z^{s}([0,T])$ by the following norm:
\begin{eqnarray*}
\|u\|_{Z^{s}([0,T])}:={\rm inf}\left\{\|v\|_{Z^{s}}:\qquad u(t)=v(t)
\qquad on \qquad t\in I\right\}.
\end{eqnarray*}

The main result of this paper are as follow.

\begin{Theorem}\label{Thm1}
Let $s<\frac{1}{2}$,
\begin{eqnarray}
F(u_{1},u_{2})=\frac{1}{2}\partial_{x}(u_{1}u_{2})+\partial_{x}(1-\partial_{x}^{2})^{-1}\left[u_{1}u_{2}+\frac{1}{2}(\partial_{x}u_{1})(\partial_{x}u_{2})\right],\label{1.06}
\end{eqnarray}
and $u_{j}(j=1,2)$ be $2\pi$-periodic
functions.
Then
\begin{eqnarray*}
\left\|\mathscr{F}^{-1}\left[\langle\tau-k^{3}\rangle^{-1}\mathscr{F}F(u_{1},u_{2})\right]\right\|_{W^{s}}\leq C\prod_{j=1}^{2}\|u_{j}\|_{W^{s}}
\end{eqnarray*}
is invalid.
\end{Theorem}

\begin{Theorem}\label{Thm2}
Let $\frac{1}{6}<s<\frac{1}{2}$, $\int_{\mathbf{T}} u_{0}dx=0$   and $u_{0}$ be $2\pi\lambda$-periodic
function.
Then the Cauchy problems (\ref{1.01})(\ref{1.02})
are locally well-posed in $H^{s}(\mathbf{T})$ for small initial data.
\end{Theorem}

\noindent {\bf Remark 1.} From Theorem 1.2,  we know that Theorem 1.1 does not imply ill-posedness of (\ref{1.01}) for $s<\frac{1}{2}.$ We will pursue the optimal regularity indices for (\ref{1.01}).

\noindent {\bf Remark 2.} We make a commentary  on the particular
choice of parameters in the $Z^{s}$ function space. Combining  Lemma 2.4  with (2.4) of Lemma 2.3, $\Omega_{j}(1\leq j\leq8),$   we know that  choosing function space  $X_{s,\frac{5}{6}-\epsilon}$ in regions $D_{1}\cup D_{5}$ is sufficient due to Lemma 2.7.
Combining high$\times$ high$\rightarrow$ low interaction with $\Omega_{j}(1\leq j\leq8)$, we know that  choosing function space $X_{1-s,s}$ in regions $D_{3}\cup D_{4}$ is sufficient
due to Lemma 2.7. By using a direct computation and fixing spaces $X_{s,\frac{5}{6}-\epsilon}$ closely related to $D_{1}\cup D_{5}$ and $X_{1-s,s}$  closely related to $D_{3}\cup D_{4}$, we know that $X_{1-s,s+\frac{1}{3}-\epsilon}$
is the suitable function space for the region $D_{2}$ in view of Lemma 2.7.

The rest of the paper is arranged as follows. In Section 2,  we give some
preliminaries. In Section 3, we establish some important  bilinear estimates. In Section 4,
we give the proof of Theorem 1.1.
In Section 5, we give the proof of Theorem 1.2.

\bigskip
\bigskip
\bigskip

 \noindent{\large\bf 2. Preliminaries}

\setcounter{equation}{0}

\setcounter{Theorem}{0}

\setcounter{Lemma}{0}

\setcounter{section}{2}

In this section, we give some preliminaries which are crucial in establishing
 Lemmas \ref{Lemma3.1}-3.5 and Theorems 1.1,1.2.

\begin{Lemma}\label{Lemma2.1}
Let $u_{l}$ with $l=1,2$ be $L^{2}([0,2\pi)\times \R)$-real valued functions.
Then for any $(l_{1},l_{2}) \in \N^{2}$,
\begin{eqnarray}
      \left\|(\Psi_{l_{1}}u_{1})*(\Psi_{l_{2}}u_{2})\right\|_{L_{k\tau}^{2}}
      \leq C\left(2^{l_{1}}
      \wedge2^{l_{2}}\right)^{\frac{1}{2}}\left(2^{l_{1}}\vee2^{l_{2}}\right)^{\frac{1}{6}}
      \|\Psi_{l_{1}}u_{1}\|_{L^{2}}\|\Psi_{l_{2}}u_{2}\|_{L^{2}}.
        \label{2.01}
\end{eqnarray}
\end{Lemma}

For the proof of  Lemma 2.1, we refer the readers to Lemma A.1 of \cite{Molinet,LYLH}.

\begin{Lemma}\label{Lemma2.2}
Let $u(x,t),v(x,t)$ be  $2\pi$-periodic functions and $a+b\geq \frac{2}{3}$
and ${\rm min}\{a,b\}>\frac{1}{6}$. Then, we have that
\begin{eqnarray}
     && \left\|uv\right\|_{L_{xt}^{2}}\leq C\|u\|_{X_{0,a}([0,2\pi) \times \SR)}
     \|v\|_{X_{0,b}([0,2\pi) \times \SR)},
       \label{2.02}\\
       &&\left\|uv\right\|_{X_{0,-a}}\leq C\|u\|_{X_{0,b}([0,2\pi) \times \SR)}
       \|v\|_{L_{xt}^{2}}.
       \label{2.03}
\end{eqnarray}
\end{Lemma}
\begin{proof}From Lemma 2.1, we have that
\begin{eqnarray*}
&&\|uv\|_{L_{xt}^{2}}\leq C \sum_{l_{1}\geq 0}\sum_{l_{2}\geq 0}
\left\|(\Psi_{l_{1}}u)*(\Psi_{l_{2}}v)\right\|_{L_{xt}^{2}}\nonumber\\
&&\leq C\sum_{l_{1}\geq 0}\sum_{l_{2}\geq 0}\left(2^{l_{1}}
      \wedge2^{l_{2}}\right)^{1/2}\left(2^{l_{1}}\vee2^{l_{2}}\right)^{\frac{1}{6}}
      \|\Psi_{l_{1}}u\|_{L^{2}}\|\Psi_{l_{2}}v\|_{L^{2}}  .
\end{eqnarray*}
Let $M_{j}=2^{l_{j}}$  with  $j=1,2.$ Without loss of generality, we
can assume that $M_{1}\geq M_{2}$
and $M_{1}=NM_{2}$ and $u_{M_{1}}=\Psi_{l_{1}}u$ and $v_{M_{2}}=\Psi_{l_{2}}v$, then we have that
\begin{eqnarray*}
&&\|uv\|_{L_{xt}^{2}}\leq  C\sum_{M_{1},M_{2}\geq 1}M_{1}^{\frac{1}{6}}
      M_{2}^{1/2}\|u_{M_{1}}\|_{L^{2}}\|v_{M_{2}}\|_{L^{2}}\nonumber\\
      &&\leq C\sum_{N,M_{2}\geq 1}M_{2}^{\frac{2}{3}}N^{\frac{1}{6}}
      \|u_{NM_{2}}\|_{L^{2}}\|v_{M_{2}}\|_{L^{2}}\nonumber\\
      &&\leq C\sum_{M_{2},N\geq 1}N^{\frac{1}{6}-a}M_{2}^{\frac{2}{3}-a-b}
      (NM_{2})^{a}\|u_{NM_{2}}\|_{L^{2}}M_{2}^{b}\|v_{M_{2}}\|_{L^{2}}\nonumber\\
      &&\leq C\|u\|_{X_{0,a}([0,2\pi) \times \SR)}\|v\|_{X_{0,b}([0,2\pi) \times \SR)}.
\end{eqnarray*}
We can derive (\ref{2.03})  by duality.
\end{proof}

We have completed the proof of Lemma 2.2.

\begin{Lemma}\label{Lemma2.3}
Let $u(x,t),v(x,t)$ be  $2\pi\lambda$-periodic functions and $a+b\geq \frac{2}{3}$
and ${\rm min}\{a,b\}>\frac{1}{6}$. Then
\begin{eqnarray}
     && \left\|uv\right\|_{L_{xt}^{2}}\leq C\|u\|_{X_{0,a}(\mathbf{T} \times \SR)}
     \|v\|_{X_{0,b}(\mathbf{T} \times \SR)},
       \label{2.04}\\
       &&\left\|uv\right\|_{X_{0,-a}}\leq C\|u\|_{X_{0,b}(\mathbf{T} \times \SR)}
       \|v\|_{L_{xt}^{2}}.
       \label{2.05}
\end{eqnarray}
\end{Lemma}
\begin{proof}By using a similar technique of Lemma 3.4 in \cite{Molinet} and
Lemma 2.2, we  obtain Lemma 2.3.
\end{proof}

We have completed the proof of Lemma 2.3.

\begin{Lemma}\label{Lemma2.4}Let $s\in \R$. Then, we have that
\begin{eqnarray*}
\left\| \eta(t)\int_{0}^{t}S(t-\tau)F(\tau)d\tau\right\|_{Z^{s}}
\leq C\|\Lambda^{-1}F\|_{Z^{s}}.
\end{eqnarray*}
\end{Lemma}

For the proof of Lemma 2.4, we refer the readers to \cite{BT}.

\begin{Lemma}\label{Lemma2.5}Let $\frac{1}{6}+\epsilon\leq s\leq\frac{1}{2}-2\epsilon$. Then, we have that
\begin{eqnarray}
&&\|u\|_{X_{s,\frac{1}{6}+\epsilon}}\leq C\|u\|_{Z^{s}}
\leq C\|u\|_{X_{s,\frac{5}{6}-\epsilon}},\label{2.06}\\
&&\|u\|_{X_{s,\frac{1}{2}}(D_{1}\bigcup D_{2})}\leq
C\|u\|_{Z^{s}(D_{1}\bigcup D_{2})}.\label{2.07}
\end{eqnarray}
\end{Lemma}
\begin{proof} We firstly prove that (\ref{2.06}). When
$\supp \mathscr{F}u\subset D_{1}\cup D_{5}$,
since $\frac{5}{6}-\epsilon\geq \frac{1}{6}+2\epsilon,$ we have that
$\|u\|_{X_{s,\frac{5}{6}-\epsilon}}\geq \|u\|_{X_{s,\frac{1}{6}+\epsilon}}$.
When $\supp \mathscr{F}u\subset D_{2}$,
 since $\frac{1}{6}+\epsilon\leq s\leq\frac{1}{2}-2\epsilon$,  we have that
$\langle \sigma\rangle ^{s+\frac{1}{6}-3\epsilon}\geq C\langle k\rangle^{2s-1}$
which yields
that $\langle k\rangle^{1-s}\langle \sigma\rangle^{s+\frac{1}{3}-\epsilon}
\geq C\langle k\rangle^{s}\langle \sigma\rangle^{\frac{1}{6}+\epsilon}$,
thus, we have that $\|u\|_{X_{1-s,s+\frac{1}{3}-\epsilon}}\geq
\|u\|_{X_{s,\frac{1}{6}+2\epsilon}}$. When $\supp \mathscr{F}u\subset D_{3}\cup D_{4}$,
since $\frac{1}{6}+\epsilon\leq s\leq\frac{1}{2}-2\epsilon$,  we have that
$\langle \sigma\rangle ^{s-\frac{1}{6}-2\epsilon}\geq C\langle k\rangle^{2s-1}$ which yields
that $\langle k\rangle^{1-s}\langle \sigma\rangle^{s}\geq C\langle k\rangle^{s}\langle
\sigma\rangle^{\frac{1}{6}+2\epsilon}$,
thus, we have that $\|u\|_{X_{1-s,s}}\geq \|u\|_{X_{s,\frac{1}{6}+2\epsilon}}$.
Consequently, we have that $\|u\|_{Z^{s}}\geq C\|u\|_{X_{s,\frac{1}{6}+\epsilon}}.$
When $\supp \mathscr{F}u\subset D_{2}$,
 since $\frac{1}{6}+\epsilon\leq s\leq\frac{1}{2}-2\epsilon$,  we have that
$\langle \sigma\rangle ^{\frac{1}{2}-s}\geq C\langle k\rangle^{1-2s}$ which yields
that $\langle k\rangle^{1-s}\langle \sigma\rangle^{s+\frac{1}{3}-\epsilon}
\leq C\langle k\rangle^{s}
\langle \sigma\rangle^{\frac{5}{6}-\epsilon}$,
thus, we have that $\|u\|_{X_{1-s,s+\frac{1}{3}-\epsilon}}
\leq C\|u\|_{X_{s,\frac{5}{6}-\epsilon}}$.
When $\supp \mathscr{F}u\subset D_{3}\cup D_{4}$,
since $\frac{1}{6}+\epsilon\leq s\leq\frac{1}{2}-2\epsilon$,  we have that
$\langle \sigma\rangle ^{\frac{5}{6}-\epsilon-s}\geq
C\langle k\rangle^{1-2s}$ which yields
that $\langle k\rangle^{1-s}\langle \sigma\rangle^{s}
\leq C\langle k\rangle^{s}\langle \sigma\rangle^{\frac{5}{6}-\epsilon}$,
thus, we have that $\|u\|_{X_{s,\frac{5}{6}-\epsilon}}\geq C \|u\|_{X_{s,1-s}}$.
Consequently, we have that $\|u\|_{Z^{s}}\leq C\|u\|_{X_{s,\frac{5}{6}-\epsilon}}.$
By Cauchy-Schwartz inequality with respect to $\tau$, we have that
$\|\langle k\rangle^{s}\mathscr{F}u\|_{l_{k}^{2}l_{\tau}^{1}}\leq C
\|u\|_{X_{s,\frac{5}{6}-\epsilon}},$  consequently, we have that $\|u\|_{Z^{s}}
\leq C\|u\|_{X_{s,\frac{5}{6}-\epsilon}}.$  Now we prove (\ref{2.07}). When
$\supp \mathscr{F}u\subset D_{1}$, since $\frac{5}{6}-\epsilon\geq \frac{1}{2},$
we have that
$\|u\|_{X_{s,\frac{5}{6}-\epsilon}}\geq \|u\|_{X_{s,\frac{1}{2}}}$. When
$\supp \mathscr{F}u\subset D_{2}$, since $s\geq\epsilon+\frac{1}{6}$, we have that
$\langle k\rangle^{s}\langle\sigma\rangle^{1/2}\leq C\langle k\rangle
^{1-s}\langle\sigma\rangle^{s+\frac{1}{3}-\epsilon},$ consequently, we have that
$\|u\|_{X_{1-s,s+\frac{1}{3}-\epsilon}}\geq \|u\|_{X_{s,\frac{1}{2}}}$.
\end{proof}

We have completed the proof of Lemma 2.5.

\begin{Lemma}\label{Lemma2.6}
Assume that  $s\in \R$. Then, we have that
\begin{eqnarray*}
\left\| \eta(t)S(t)\phi\right\|_{Z^{s}}\leq C\|\phi\|_{H^{s}(\mathbf{T})}.
\end{eqnarray*}
\end{Lemma}
\begin{proof} From  Lemma 2.5, we have
that $X_{s,\frac{5}{6}-\epsilon}
\hookrightarrow Z^{s}\hookrightarrow C([0,T]:H^{s}(\mathbf{T})).$
\end{proof}

We have completed the proof of Lemma 2.6.
\begin{Lemma}\label{Lemma2.7}
Let $\tau=\tau_{1}+\tau_{2},k=k_{1}+k_{2},$
$
\sigma=\tau-k^{3},\sigma_{1}= \tau_{1}-k_{1}^{3},\sigma_{2}= \tau_{1}-k_{2}^{3}.
$
 Then, we have that
\begin{eqnarray*}
3{\rm max}\left\{|\sigma|,|\sigma_{1}|,|\sigma_{2}|\right\}
\geq|\sigma-\sigma_{1}-\sigma_{2}|=
\left|k^{3}-k_{1}^{3}-k_{2}^{3}\right|= 3|kk_{1}k_{2}|.
\end{eqnarray*}
Moreover, one of the following three cases  must occurs:
\begin{eqnarray*}
 &&(a): |\sigma|={\rm max}
\left\{|\sigma|,|\sigma_{1}|,|\sigma_{2}|\right\}\geq |kk_{1}k_{2}|,\nonumber\\
&& (b): |\sigma_{1}|={\rm max}
\left\{|\sigma|,|\sigma_{1}|,|\sigma_{2}|\right\}\geq |kk_{1}k_{2}|,\nonumber\\
&& (c): |\sigma_{2}|={\rm max}
\left\{|\sigma|,|\sigma_{1}|,|\sigma_{2}|\right\}\geq |kk_{1}k_{2}|.
\end{eqnarray*}
\end{Lemma}

Since the proof is easy and we omit it here.

\bigskip
\bigskip

\noindent{\large\bf 3. Bilinear estimates }

\setcounter{equation}{0}

 \setcounter{Theorem}{0}

\setcounter{Lemma}{0}

 \setcounter{section}{3}
 In this section, we establish Lemmas 3.1-3.5.
 \begin{Lemma}\label{Lemma3.1}
Let  $\frac{1}{6}+\epsilon\leq s\leq\frac{1}{2}-2\epsilon$,
where $0<\epsilon<\frac{1}{10^{9}}.$ Then, we have
\begin{eqnarray}
      \left\|\Lambda^{-1}\partial_{x}(1-\partial_{x}^{2})^{-1}\prod_{j=1}^{2}
      (\partial_{x}u_{j})\right\|_{X^{s}}\leq C\prod\limits_{j=1}^{2}\|u_{j}\|_{Z^{s}},
        \label{3.01}
\end{eqnarray}
here $C>0$,  which  is  independent of  $\lambda$,
$\left\|\cdot\right\|_{X^{s}}$ is the
norm removing $\left\|\cdot\right\|_{Y^{s}}$ from
$\left\|\cdot\right\|_{Z^{s}}.$
\end{Lemma}
{\bf Proof.} Obviously, $\left(\R\times\dot{Z}_{\lambda}\right)^{2}
\subset \bigcup\limits_{j=1}^{8}\Omega_{j},$
where
\begin{eqnarray*}
&&\Omega_{1}=\left\{(\tau_{1},k_{1},\tau,k)\in \left(\R\times\dot{Z_{\lambda}}\right)^{2}:
{\rm max}\left\{|k_{1}|, |k|\right\}\leq1\right\},\\
&&\Omega_{2}=\left\{(\tau_{1},k_{1},\tau,k)\in \left(\R\times\dot{Z_{\lambda}}\right)^{2}\cap
\Omega_{1}^{c}:|k_{1}|\sim |k_{2}|\gg |k|\geq1\right\},\\
&&\Omega_{3}=\left\{(\tau_{1},k_{1},\tau,k)\in \left(\R\times\dot{Z_{\lambda}}\right)^{2}\cap
\Omega_{1}^{c}:|k_{1}|\sim |k_{2}|\gg |k|,1\geq|k|\geq\frac{1}{\lambda}\right\},\\
&&\Omega_{4}=\left\{(\tau_{1},k_{1},\tau,k)\in \left(\R\times\dot{Z_{\lambda}}\right)^{2}\cap
\Omega_{1}^{c}:|k|\sim |k_{2}|\gg |k_{1}|\geq 1\right\},\\
&&\Omega_{5}=\left\{(\tau_{1},k_{1},\tau,k)\in \left(\R\times\dot{Z_{\lambda}}\right)^{2}\cap
\Omega_{1}^{c}:|k|\sim |k_{2}|\gg |k_{1}|,1\geq|k_{1}|\geq \frac{1}{\lambda}\right\},\\
&&\Omega_{6}=\left\{(\tau_{1},k_{1},\tau,k)\in \left(\R\times\dot{Z_{\lambda}}\right)^{2}\cap
\Omega_{1}^{c}:|k|\sim |k_{1}|\gg |k_{2}|\geq 1\right\},\\
&&\Omega_{7}=\left\{(\tau_{1},k_{1},\tau,k)\in \left(\R\times\dot{Z_{\lambda}}\right)^{2}\cap
\Omega_{1}^{c}:|k|\sim |k_{1}|\gg |k_{2}|,1\geq|k_{2}|\geq \frac{1}{\lambda}\right\},\\
&&\Omega_{8}=\left\{(\tau_{1},k_{1},\tau,k)\in \left(\R\times\dot{Z_{\lambda}}\right)^{2}\cap
\Omega_{1}^{c}:|k|\sim |k_{1}|\sim |k_{2}|\geq 1\right\}.
\end{eqnarray*}
(1) In region $\Omega_{1}$.
By using (\ref{2.06}), since ${\rm max}\left\{|k_{1}|,|k|\right\}\leq 1,$ by
using the Cauchy-Schwartz inequality
and  the Young inequality as well as Lemma 2.5,   we have that
\begin{eqnarray*}
&&\left\|\Lambda^{-1}\partial_{x}
(1-\partial_{x}^{2})^{-1}\prod_{j=1}^{2}
(\partial_{x}u_{j})\right\|_{X^{s}}\leq
C\left\|\Lambda^{-1}\partial_{x}(1-\partial_{x}^{2})^{-1}
\prod_{j=1}^{2}(\partial_{x}u_{j})\right\|_{X_{s,\frac{5}{6}-\epsilon}}\nonumber\\
&&\leq C\left\||k|\langle \sigma\rangle^{-\frac{1}{6}-\epsilon}
\left(\mathscr{F}u_{1}*\mathscr{F}u_{2}\right)
\right\|_{l_{k}^{2}L_{\tau}^{2}}\nonumber\\
&&\leq C\||k|\|_{l_{k}^{2}}\left\|\mathscr{F}u_{1}*\mathscr{F}u_{2}
\right\|_{l_{k}^{\infty}L_{\tau}^{2}}\leq C\|\mathscr{F}u_{1}\|_{
l_{k}^{2}L_{\tau}^{2}}\|\mathscr{F}u_{2}\|_{l_{k}^{2}L_{\tau}^{1}}\leq
C\|u_{1}\|_{X_{s,\frac{1}{6}+\epsilon}}\|u_{2}\|_{Y^{s}}\nonumber
\\&&\leq C\prod_{j=1}^{2}\|u_{j}\|_{Z^{s}}.
\end{eqnarray*}
(2) In region $\Omega_{2}$. In this  region, we consider (a)-(c) of Lemma 2.7, respectively.

\noindent(a) Case $|\sigma|={\rm max}\left\{|\sigma|,|\sigma_{1}|,|\sigma_{2}|\right\}.$
In this case, we have that
$\supp \left[\mathscr{F}u_{1}*\mathscr{F}u_{2}\right]\!\subset\! D_{3}.$

\noindent When $\supp \mathscr{F}u_{j}$ $ \subset D_{1}\cup D_{2}$ with $j=1,2$,
 by using
   Lemmas 2.5, 2.7,  2.3, since $\frac{1}{6}+\epsilon\leq s\leq\frac{1}{2}-2\epsilon$,
 we have that
\begin{eqnarray*}
&&\left\|\Lambda^{-1}\partial_{x}(1-\partial_{x}^{2})^{-1}\prod_{j=1}^{2}
(\partial_{x}u_{j})\right\|_{X^{s}}\leq C\left\|\langle k\rangle^{-s}
\langle\sigma\rangle^{s-1}\left[(\langle k\rangle\mathscr{F}u_{1})*
(\langle k\rangle\mathscr{F}u_{2})\right]\right\|_{l_{k}^{2}L_{\tau}^{2}}\nonumber\\
&&\leq C\|(J^{s}u_{1})(J^{s}u_{2})\|_{L_{xt}^{2}}\leq C
\|u_{1}\|_{X_{s,\frac{1}{2}}}\|u_{2}\|_{X_{s,\frac{1}{6}+\epsilon}}
\leq C\prod_{j=1}^{2}\|u_{j}\|_{Z^{s}}.
\end{eqnarray*}
When $\supp \mathscr{F}u_{1} \subset D_{3}$,
 by using   Lemma  2.7, the Young inequality
   and the  Cauchy-Schwartz inequality,
   since $\frac{1}{6}+\epsilon\leq s\leq\frac{1}{2}-2\epsilon$, we have that
\begin{eqnarray*}
&&\left\|\Lambda^{-1}\partial_{x}(1-\partial_{x}^{2})^{-1}
\prod_{j=1}^{2}(\partial_{x}u_{j})\right\|_{X^{s}}
\leq C\left\|\langle k\rangle^{-s}\langle\sigma\rangle^{s-1}
\left[\langle k\rangle\mathscr{F}u_{1}*\langle k\rangle
\mathscr{F}u_{2}\right]\right\|_{l_{k}^{2}L_{\tau}^{2}}\nonumber\\
&&\leq C\left\|\left[\langle k\rangle^{2s}\mathscr{F}u_{1}\right]*
\mathscr{F}u_{2}\right\|_{l_{k}^{2}L_{\tau}^{2}}\nonumber\\
&&\leq C\left\|\left[\langle k\rangle^{1-s}
\langle\sigma\rangle^{s}\mathscr{F}u_{1}\right]*
\left[\langle k\rangle^{-1}\mathscr{F}u_{2}\right]
\right\|_{l_{k}^{2}L_{\tau}^{2}}
\nonumber\\
&&\leq C\|u_{1}\|_{X_{1-s,s}}\|\langle k\rangle^{-1}
\mathscr{F}u_{2}\|_{l_{k}^{1}L_{\tau}^{1}}
\leq C\|u\|_{X_{1-s,s}}\|\langle k\rangle^{s}
\mathscr{F}u_{2}\|_{l_{k}^{2}L_{\tau}^{1}}\leq
C\prod_{j=1}^{2}\|u_{j}\|_{Z^{s}}.
\end{eqnarray*}
When $\supp \mathscr{F}u_{2} \subset D_{3}$,
this case can be proved similarly to
$\supp \mathscr{F}u_{1} \subset D_{3}$.

\noindent (b) Case $|\sigma_{1}|={\rm max}\left\{|\sigma|,|\sigma_{1}|,|\sigma_{2}|\right\},$
 we consider the following cases:
\begin{eqnarray*}
(i): |\sigma_{1}|>4{\rm max}\left\{|\sigma|,|\sigma_{2}|\right\},
(ii):|\sigma_{1}|\leq4{\rm max}\left\{|\sigma|,|\sigma_{2}|\right\},
\end{eqnarray*}
respectively.

\noindent
When (i) occurs:
if $\supp  \mathscr{F}u_{1}\subset D_{1}$ which yields that $1\leq|k|\leq C$,
by using  Lemmas 2.5, 2.7, 2.3, since $\frac{1}{6}+\epsilon\leq s\leq\frac{1}{2}-2\epsilon,$ we have that
\begin{eqnarray*}
&&\left\|\Lambda^{-1}\partial_{x}(1-\partial_{x}^{2})^{-1}
\prod_{j=1}^{2}(\partial_{x}u_{j})\right\|_{X^{s}}
\leq C\left\|\langle k\rangle ^{s-1}\langle \sigma \rangle
^{-\frac{1}{6}-\epsilon}
(\langle k\rangle\mathscr{F}u_{1})*(\langle k\rangle
\mathscr{F}u_{2})\right\|_{l_{k}^{2}L_{\tau}^{2}}\nonumber\\&&
\leq C\left\|\langle \sigma \rangle ^{-\frac{1}{6}-\epsilon}
\left[(\langle k\rangle^{s}\langle\sigma\rangle^{\frac{5}{6}-\epsilon}
\mathscr{F}u_{1})*(\langle k\rangle ^{2\epsilon+\frac{1}{3}-s}
\mathscr{F}u_{2})\right]\right\|_{l_{k}^{2}L_{\tau}^{2}}\nonumber\\
&&\leq C\left\|\left(J^{s}\Lambda ^{\frac{5}{6}-\epsilon}u_{1}\right)
\left(J^{-s+\frac{1}{3}+2\epsilon}u_{2}\right)\right\|_{X_{0,-\frac{1}{6}-\epsilon}}\nonumber\\
&&\leq C\|u_{1}\|_{X_{s,\frac{5}{6}-\epsilon}}\|u_{2}\|_{X_{s,\frac{1}{2}}}\leq
C\|u_{1}\|_{X_{s,\frac{5}{6}-\epsilon}}\|u_{2}\|_{X_{s,\frac{1}{2}}}\leq
C\prod_{j=1}^{2}\|u_{j}\|_{Z^{s}};
\end{eqnarray*}
if $\supp \mathscr{F} u_{1}\subset D_{2},$
by using Lemmas 2.5, 2.7, 2.3, since $\frac{1}{6}+\epsilon\leq s\leq\frac{1}{2}-2\epsilon$,
we have that
\begin{eqnarray*}
&&\left\|\Lambda^{-1}\partial_{x}(1-\partial_{x}^{2})^{-1}
\prod_{j=1}^{2}(\partial_{x}u_{j})\right\|_{X^{s}}
\leq C\left\|\langle k\rangle ^{s-1}\langle \sigma \rangle ^{-\frac{1}{6}-\epsilon}
\left[(\langle k\rangle\mathscr{F}u_{1})*(\langle k\rangle\mathscr{F}u_{2})\right]
\right\|_{l_{k}^{2}L_{\tau}^{2}}\nonumber\\&&
\leq C\left\|\left(J^{1-s}\Lambda ^{s+\frac{1}{3}-\epsilon}u_{1}\right)
\left(J^{-s+\frac{1}{3}+2\epsilon}u_{2}\right)
\right\|_{X_{0,-\frac{1}{6}-\epsilon}}\nonumber\\
&&\leq C\|u_{1}\|_{X_{1-s,s+\frac{1}{3}-\epsilon}}
\|u_{2}\|_{X_{-s+\frac{1}{3}+2\epsilon,\frac{1}{2}}}\leq
 C\|u_{1}\|_{X_{1-s,s+\frac{1}{3}-\epsilon}}\|u_{2}\|_{X_{s,\frac{1}{2}}}\leq
C\prod_{j=1}^{2}\|u_{j}\|_{Z^{s}}.
\end{eqnarray*}
When (ii) occurs: we have $|\sigma_{1}|\sim |\sigma|$ or $|\sigma_{1}|\sim |\sigma_{2}|$.

\noindent
Case $|\sigma_{1}|\sim |\sigma|$   can be proved similarly to
$|\sigma|={\rm max}\left\{|\sigma|,|\sigma_{1}|,|\sigma_{2}|\right\}.$

\noindent
When $|\sigma_{1}|\sim |\sigma_{2}|$, we consider $\supp \mathscr{F}u_{1}\subset D_{1}$,
$\supp\mathscr{F} u_{1}\subset D_{2}$, $\supp\mathscr{F} u_{1}\subset D_{3}$, respectively.

\noindent When $\supp\mathscr{F} u_{1}\subset D_{1}$
which yields that $1\leq|k|\leq C$,
by using  Lemmas 2.5, 2.7, 2.3,
since $\frac{1}{6}+\epsilon\leq s\leq\frac{1}{2}-2\epsilon,$
 we have that
\begin{eqnarray*}
&&\left\|\Lambda^{-1}\partial_{x}(1-\partial_{x}^{2})^{-1}\prod_{j=1}^{2}
(\partial_{x}u_{j})\right\|_{X^{s}}\leq C
\left\|\langle k\rangle ^{s-1}\langle \sigma \rangle ^{-\frac{1}{6}-\epsilon}
(\langle k\rangle\mathscr{F}u_{1})*(\langle k\rangle\mathscr{F}u_{2})
\right\|_{l_{k}^{2}L_{\tau}^{2}}\nonumber\\&&
\leq C\left\|(\langle k\rangle^{s}\langle \sigma \rangle ^{\frac{5}{6}-\epsilon}
\mathscr{F}u_{1})*(\langle k\rangle ^{-s+\frac{1}{3}+2\epsilon}\mathscr{F}u_{2})\right\|_{X_{0,-\frac{1}{6}-\epsilon}}\nonumber\\
&&\leq C\left\|\left(J^{s}\Lambda ^{\frac{5}{6}-\epsilon}u_{1}\right)
\left(J^{-s+\frac{1}{3}+2\epsilon}u_{2}\right)\right\|_{
X_{0,-\frac{1}{6}-\epsilon}}\nonumber\\
&&\leq C\|u_{1}\|_{X_{s,\frac{5}{6}-\epsilon}}\|u_{2}\|_{
X_{-s+\frac{1}{3}+2\epsilon,\frac{1}{2}}}\leq C\|u_{1}\|_{
X_{s,\frac{5}{6}-\epsilon}}\|u_{2}\|_{X_{s,\frac{1}{2}}}\leq
C\prod_{j=1}^{2}\|u_{j}\|_{Z^{s}}.
\end{eqnarray*}
When $\supp \mathscr{F}u_{1}\subset D_{2}$,
by using
the  Lemmas 2.5, 2.7, 2.3, since $\frac{1}{6}+\epsilon\leq s\leq\frac{1}{2}-2\epsilon$,
 we have that
\begin{eqnarray*}
&&\left\|\Lambda^{-1}\partial_{x}(1-\partial_{x}^{2})^{-1}\prod_{j=1}^{2}
(\partial_{x}u_{j})\right\|_{X^{s}}\leq C\left\|\langle k\rangle ^{s-1}
\langle \sigma \rangle ^{-\frac{1}{6}-\epsilon}
\left[(\langle k\rangle\mathscr{F}u_{1})*(\langle k\rangle\mathscr{F}u_{2})\right]
\right\|_{l_{k}^{2}L_{\tau}^{2}}\nonumber\\&&
\leq C\left\|\left(J^{1-s}\Lambda ^{s+\frac{1}{3}-\epsilon}u_{1}\right)
\left(J^{-s+\frac{1}{3}+2\epsilon}u_{2}\right)
\right\|_{X_{0,-\frac{1}{6}-\epsilon}}\nonumber\\
&&\leq C\|u_{1}\|_{X_{1-s,s+\frac{1}{3}-\epsilon}}\|u_{2}\|_{X_{-s+\frac{1}{3}
+2\epsilon,\frac{1}{2}}}\nonumber\\
&&\leq C\|u_{1}\|_{X_{1-s,s+\frac{1}{3}-\epsilon}}\|u_{2}\|_{X_{s,\frac{1}{2}}}\leq
C\prod_{j=1}^{2}\|u_{j}\|_{Z^{s}}.
\end{eqnarray*}
Case $\supp\mathscr{F} u_{1}\subset D_{3}$,
by using  Lemma 2.5, the H\"older
inequality  and the  Young inequality as well as Lemmas 2.7, 2.3, since
$ \frac{1}{6}+\epsilon\leq s\leq\frac{1}{2}-2\epsilon,$
 we  have that
\begin{eqnarray*}
&&\left\|\Lambda^{-1}\partial_{x}(1-\partial_{x}^{2})^{-1}\prod_{j=1}^{2}
(\partial_{x}u_{j})\right\|_{X^{s}}\leq C
\left\|\langle k\rangle ^{s-1}\langle \sigma\rangle ^{-\frac{1}{6}-\epsilon}
\left[(|k|\mathscr{F}u_{1})*(|k|\mathscr{F}u_{2})\right]\right\|_{l_{k}^{2}L_{\tau}^{2}}\nonumber\\
&&\leq C\left\|\langle k\rangle ^{s-\frac{1}{2}+\epsilon}\langle
 \sigma \rangle^{\frac{1}{3}}
\left[(|k|\mathscr{F}u_{1})*(|k|\mathscr{F}u_{2})\right]
\right\|_{l_{k}^{\infty}L_{\tau}^{\infty}}\nonumber\\
&&\leq C\left\|\langle k\rangle ^{s-\frac{1}{2}+\epsilon}\left(\langle k\rangle^{2}\mathscr{F}u_{1})
*(\langle k\rangle\mathscr{F}u_{2}\right)\right\|_{l_{k}^{\infty}l_{\tau}^{\infty}}\nonumber\\&&
\leq C\left\|\left(\langle k\rangle ^{s+\frac{5}{2}+\epsilon}\mathscr{F}u_{1}
\right)*\mathscr{F}u_{2}\right\|_{l_{k}^{\infty}l_{\tau}^{\infty}}\nonumber\\&&
\leq C
\left\|\langle k\rangle ^{-3s+\frac{1}{2}+\epsilon}\right\|_{l_{k}^{\infty}}
\prod_{j=1}^{2}\|u_{j}\|_{X_{1-s,s}}\leq C
\prod_{j=1}^{2}\|u_{j}\|_{X_{1-s,s}}\leq C\prod_{j=1}^{2}\|u_{j}\|_{Z^{s}}.
\end{eqnarray*}
(c) Case $|\sigma_{2}|={\rm max }\left\{|\sigma|,|\sigma_{1}|,|\sigma_{2}|\right\}.$
This case can be proved similarly to case (b).

\noindent (3) Region $\Omega_{3}$.
We  consider $|k|\leq |k_{1}|^{-2}$  and  $|k_{1}|^{-2}<|k|\leq 1,$ respectively.

\noindent When $|k|\leq |k_{1}|^{-2}$, by using  the H\"older
inequality and the Young inequality as well as  Lemma 2.5,
since $\frac{1}{6}+\epsilon \leq s\leq\frac{1}{2}-2\epsilon,$ we have that
\begin{eqnarray*}
&&\left\|\Lambda^{-1}\partial_{x}(1-\partial_{x}^{2})^{-1}\prod_{j=1}^{2}
(\partial_{x}u_{j})\right\|_{X^{s}}\leq C\left\|\langle k\rangle^{s-2}|k|\langle \sigma\rangle ^{-\frac{1}{6}-\epsilon}
\left[(|k|\mathscr{F}u_{1})*(|k|\mathscr{F}u_{2})\right]\right\|_{l_{k}^{2}L_{\tau}^{2}}\nonumber\\
&&\leq C\left\|\left[\mathscr{F}u_{1}*\mathscr{F}u_{2}\right]
\right\|_{l_{k}^{\infty}L_{\tau}^{2}}
\leq C\|u_{1}\|_{X_{0,0}}\|u_{2}\|_{Y^{0}}\nonumber\\
&&\leq C\|u_{1}\|_{X_{s,\frac{1}{6}}}\|u_{2}\|_{Y^{s}}
\leq C\prod_{j=1}^{2}\|u_{j}\|_{Z^{s}}.
\end{eqnarray*}
Now we consider the case $|k_{1}|^{-2}<|k|\leq 1.$
In this  case, we consider (a)-(c) of Lemma 2.7, respectively.

\noindent When   (a) occurs:
$\supp \left[\mathscr{F}u_{1}*\mathscr{F}u_{2}\right]\subset D_{4}$,
by using  Lemma 2.7  and the Young inequality,   we have that
\begin{eqnarray*}
&&\left\|\Lambda^{-1}\partial_{x}(1-\partial_{x}^{2})^{-1}
\prod_{j=1}^{2}(\partial_{x}u_{j})\right\|_{X^{s}}\leq C
\left\||k|\langle k\rangle ^{-s-1}\langle \sigma \rangle ^{s-1}\left[(|k|\mathscr{F}u_{1})*(|k|\mathscr{F}u_{2})\right]
\right\|_{l_{k}^{2}L_{\tau}^{2}}\nonumber\\
&&\leq C\left\| \left[(|k|^{s}\mathscr{F}u_{1})*(|k|^{s}\mathscr{F}u_{2})\right]
\right\|_{l_{k}^{2}L_{\tau}^{2}}\nonumber\\
&&\leq C\left\|\left[(|k|^{s}\mathscr{F}u_{1})*(|k|^{s}\mathscr{F}u_{2})\right]
\right\|_{l_{k}^{\infty}L_{\tau}^{2}}\nonumber\\
&&\leq C\|u_{1}\|_{X_{s,0}}\|u_{2}\|_{Y^{s}}\leq C\|u_{1}\|_{X_{s,\frac{1}{6}}}\|u_{2}\|_{Y^{s}}\leq C\prod_{j=1}^{2}\|u_{j}\|_{Z^{s}}.
\end{eqnarray*}
When (b)  occurs:  we   consider  $|\sigma_{1}|>4{\rm max}\left\{|\sigma|,|\sigma_{2}|\right\}$  and
$|\sigma_{1}|\leq4{\rm max}\left\{|\sigma|,|\sigma_{2}|\right\}$, respectively.

\noindent
When $|\sigma_{1}|>4{\rm max}\left\{|\sigma|,|\sigma_{2}|\right\}$,
 we have that $\supp \mathscr{F}u_{1} \subset D_{1}$,
 by using the Lemma 2.5, H\"older  inequality and the Young inequality,  since
$|k|\leq 1$ and $\frac{1}{6}+\epsilon \leq s\leq\frac{1}{2}-2\epsilon$,  we have that
\begin{eqnarray*}
&&   \quad
        \left\|\Lambda^{-1}\partial_{x}(1-\partial_{x}^{2})^{-1}
        \prod_{j=1}^{2}(\partial_{x}u_{j})\right\|_{X^{s}}\\
&&\leq  C\left\||k|\langle k\rangle ^{s-2}\langle \sigma \rangle ^{s-\frac{1}{6}-\epsilon}
        \left[(|k|\mathscr{F}u_{1})*(|k|\mathscr{F}u_{2})\right]\right\|_{l_{k}^{2}L_{\tau}^{2}}
        \nonumber\\
&&\leq C\left\||k|^{\frac{1}{6}+\epsilon} \left[(\langle k\rangle^{s}
       \langle\sigma\rangle^{\frac{5}{6}-\epsilon}\mathscr{F}u_{1})*
      (\langle k\rangle ^{-s+\frac{1}{3}+2\epsilon}\mathscr{F}u_{2})\right]\right\|_{l_{k}^{2}L_{\tau}^{2}}
     \nonumber\\
&&\leq C\left\| \left[(\langle k\rangle^{s}\langle\sigma\rangle^{\frac{5}{6}-\epsilon}
      \mathscr{F}u_{1})*(\langle k\rangle
   ^{-s+\frac{1}{3}+2\epsilon}\mathscr{F}u_{2})\right]\right\|_{l_{k}^{\infty}L_{\tau}^{2}}\\
&&\leq C\|u_{1}\|_{X_{s,\frac{5}{6}-\epsilon}}\|u_{2}\|_{Y^{s}}\leq C\prod_{j=1}^{2}\|u_{j}\|_{Z^{s}}.
\end{eqnarray*}
When $|\sigma_{1}|\leq4{\rm max}\left\{|\sigma|,|\sigma_{2}|\right\}$, we have
 $|\sigma_{1}|\sim |\sigma|$ or  $|\sigma_{1}|\sim |\sigma_{2}|$.

 \noindent
 Case $|\sigma_{1}|\sim|\sigma|$ can be proved similarly to
 case $|\sigma|={\rm max}\left\{|\sigma|,|\sigma_{1}|,|\sigma_{2}|\right\}.$

 \noindent
  Case $|\sigma_{1}|\sim|\sigma_{2}|$, we consider $\supp \mathscr{F}u_{1}\subset D_{1}$ and
  $\supp \mathscr{F}u_{1}\subset D_{2}$ and $\supp \mathscr{F}u_{1}\subset D_{3},$
  respectively.

\noindent   When $\supp \mathscr{F}u_{1}\subset D_{1},$
 by using Lemmas 2.5, 2.7,   2.3, since $\frac{1}{6}+\epsilon\leq s\leq\frac{1}{2}-2\epsilon,$  we have that
\begin{eqnarray*}
&&\left\|\Lambda^{-1}\partial_{x}(1-\partial_{x}^{2})^{-1}
\prod_{j=1}^{2}(\partial_{x}u_{j})\right\|_{X^{s}}
\leq C\left\||k|\langle k\rangle^{s-2}\langle \sigma \rangle^{-\frac{1}{6}-\epsilon}
\left[(|k|\mathscr{F}u_{1})*(|k|\mathscr{F}u_{2})\right]
\right\|_{l_{k}^{2}L_{\tau}^{2}}\nonumber\\
&&\leq C\left\|\left(J^{s}\Lambda^{\frac{5}{6}-\epsilon}u_{1}\right)
\left(J^{-s+\frac{1}{3}+2\epsilon}u_{2}\right)\right\|_{X_{0,-\frac{1}{6}-\epsilon}}\nonumber\\
&&\leq C\|u_{1}\|_{X_{s,\frac{5}{6}-\epsilon}}\|u_{2}\|_{X_{-s+\frac{1}{3}+2\epsilon,\frac{1}{2}}}\leq C\|u_{1}\|_{X_{s,\frac{5}{6}-\epsilon}}\|u_{2}\|_{X_{s,\frac{1}{2}}}\leq C\prod_{j=1}^{2}\|u_{j}\|_{Z^{s}}.
\end{eqnarray*}
Case $\supp \mathscr{F}u_{1}\subset D_{2}$ and $\supp\left[\mathscr{F}u_{1}*\mathscr{F}u_{2}\right]\subset D_{4}$,
 by using the H\"older inequality, since
 $\frac{1}{6}+\epsilon\leq s\leq\frac{1}{2}-2\epsilon,$ we have
\begin{eqnarray*}
&&\left\|\Lambda^{-1}\partial_{x}(1-\partial_{x}^{2})^{-1}\prod_{j=1}^{2}(\partial_{x}u_{j})
\right\|_{X^{s}}\leq C\left\||k|\langle k\rangle^{-s-1}\langle \sigma \rangle^{s-1}\left[(|k|\mathscr{F}u_{1})*(|k|\mathscr{F}u_{2})\right]
\right\|_{l_{k}^{2}L_{\tau}^{2}}\nonumber\\
&&\leq C\left\||k|\langle k\rangle^{-s-\frac{1}{2}+\varepsilon}\left[(|k|\mathscr{F}u_{1})*
(|k|\mathscr{F}u_{2})\right]\right\|_{l_{k}^{\infty}L_{\tau}^{\infty}}\nonumber\\
&&\leq C\|\langle k\rangle^{-2s-\frac{4}{3}+4\epsilon}\|_{l_{k}^{\infty}}
\prod_{j=1}^{2}\|u_{j}\|_{X_{1-s,s+\frac{1}{3}-\epsilon}}\nonumber\\
&&\leq C\prod_{j=1}^{2}\|u_{j}\|_{X_{1-s,s+\frac{1}{3}-\epsilon}}
\leq C\prod_{j=1}^{2}\|u_{j}\|_{Z^{s}};
\end{eqnarray*}
case $\supp \mathscr{F}u_{1}\subset D_{2}$
 and $\supp\left[\mathscr{F}u_{1}*\mathscr{F}u_{2}\right]\subset D_{5}$,
by using the H\"older inequality,
 since $ \frac{1}{6}+\epsilon\leq s\leq\frac{1}{2}-2\epsilon,$ we have that
\begin{eqnarray*}
&&\left\|\Lambda^{-1}(1-\partial_{x}^{2})^{-1}\partial_{x}\prod_{j=1}^{2}(\partial_{x}u_{j})
\right\|_{X^{s}}\leq C\left\||k|\langle k\rangle^{s-2}\langle \sigma \rangle^{-\frac{1}{6}-\epsilon}\left[(|k|\mathscr{F}u_{1})*(|k|
\mathscr{F}u_{2})\right]\right\|_{l_{k}^{2}L_{\tau}^{2}}\nonumber\\
&&\leq C\left\||k|\langle k\rangle^{s-\frac{3}{2}+\epsilon}\langle \sigma \rangle^{\frac{1}{3}}\left[(|k|\mathscr{F}u_{1})*
(|k|\mathscr{F}u_{2})\right]\right\|_{l_{k}^{\infty}L_{\tau}^{\infty}}\nonumber\\
&&\leq C\left\||k|^{\frac{2}{3}-s+\epsilon}\left[(\langle k\rangle^{1-s}\Lambda^{s+\frac{1}{3}-\epsilon}\mathscr{F}u_{1})*(\langle k\rangle^{-s+\frac{1}{3}+2\epsilon}
\mathscr{F}u_{2})\right]\right\|_{l_{k}^{\infty}L_{\tau}^{\infty}}\nonumber\\
&&\leq C\left\|\left[(\langle k\rangle^{1-s}\Lambda^{s+\frac{1}{3}-\epsilon}\mathscr{F}u_{1})*(\langle k\rangle^{-s+\frac{1}{3}+2\epsilon}
\mathscr{F}u_{2})\right]\right\|_{l_{k}^{\infty}L_{\tau}^{\infty}}\nonumber\\
&& \leq  C\|u_{1}\|_{X_{1-s,s+\frac{1}{3}-\epsilon}}\|u_{2}\|_{X_{s,0}}\nonumber\\
&&\leq C\prod_{j=1}^{2}\|u_{j}\|_{X_{1-s,s+\frac{1}{3}-\epsilon}}\leq
C\prod_{j=1}^{2}\|u_{j}\|_{Z^{s}}.
\end{eqnarray*}
Case $\supp \mathscr{F}u_{1}\subset D_{3}$
and $\supp\left[\mathscr{F}u_{1}*\mathscr{F}u_{2}\right]\subset D_{4}$,
by using the  H\"older inequality and the  Young  inequality,
 since $\frac{1}{6}+\epsilon\leq s\leq\frac{1}{2}-2\epsilon,$  we have that
\begin{eqnarray*}
&&\left\|\Lambda^{-1}(1-\partial_{x}^{2})^{-1}\partial_{x}
\prod_{j=1}^{2}(\partial_{x}u_{j})\right\|_{X^{s}}
\leq C\left\||k|\langle k\rangle^{-1-s}\langle \sigma \rangle^{s-1}\left[(|k|\mathscr{F}u_{1})*(|k|\mathscr{F}u_{2})\right]
\right\|_{l_{k}^{2}L_{\tau}^{2}}\nonumber\\
&&\leq C\left\||k|\langle k\rangle^{-s-\frac{1}{2}+\epsilon}\langle \sigma \rangle^{s-\frac{1}{2}+\epsilon}\left[(|k|\mathscr{F}u_{1})*(|k|\mathscr{F}u_{2})\right]
\right\|_{l_{k}^{\infty}L_{\tau}^{\infty}}\nonumber\\
&&\leq C\left\||k|\langle k\rangle^{-2s-2+4\epsilon}\left[(|k|\mathscr{F}u_{1})*(|k|\mathscr{F}u_{2})\right]
\right\|_{l_{k}^{\infty}L_{\tau}^{\infty}}\nonumber\\
&&\leq C\left\|\langle k\rangle ^{-4s}\left[(\langle k\rangle^{1-s}\Lambda ^{s}\mathscr{F}u_{1})*(\langle k\rangle^{1-s}\Lambda ^{s}\mathscr{F}u_{2})\right]
\right\|_{l_{k}^{\infty}L_{\tau}^{\infty}}\nonumber\\
&&\leq C\|\langle k\rangle^{-4s}\|_{l_{k}^{\infty}}\prod_{j=1}^{2}\|u_{j}\|_{X_{1-s,s}}\leq C\prod_{j=1}^{2}\|u_{j}\|_{X_{1-s,s}}\leq C\prod_{j=1}^{2}\|u_{j}\|_{Z^{s}};
\end{eqnarray*}
case $\supp \mathscr{F}u_{1}\subset D_{3}$
and $\supp\left[\mathscr{F}u_{1}*\mathscr{F}u_{2}\right]\subset D_{5}$,
 by  the  H\"older inequality and the  Young
inequality, since $ \frac{1}{6}+\epsilon\leq s\leq\frac{1}{2}-2\epsilon,$  we have that
\begin{eqnarray*}
&&\left\|\Lambda^{-1}(1-\partial_{x}^{2})^{-1}\partial_{x}\prod_{j=1}^{2}
(\partial_{x}u_{j})\right\|_{X^{s}}
\leq C\left\||k|\langle k\rangle^{s-2}\langle \sigma \rangle^{-\frac{1}{6}-\epsilon}\left[(|k|\mathscr{F}u_{1})*(|k|\mathscr{F}u_{2})\right]
\right\|_{l_{k}^{2}L_{\tau}^{2}}\nonumber\\
&&\leq C\left\||k|\langle k\rangle^{s-\frac{3}{2}+\epsilon}\langle \sigma \rangle^{\frac{1}{3}}\left[(|k|\mathscr{F}u_{1})*(|k|\mathscr{F}u_{2})\right]
\right\|_{l_{k}^{\infty}L_{\tau}^{\infty}}\nonumber\\
&&\leq C\left\||k|\langle k\rangle^{s-\frac{1}{2}+\epsilon}\left[(|k|\mathscr{F}u_{1})*(|k|\mathscr{F}u_{2})\right]
\right\|_{l_{k}^{\infty}L_{\tau}^{\infty}}\nonumber\\
&&\leq C\left\|\langle k\rangle ^{-4s}\left[(\langle k\rangle^{1-s}\Lambda ^{s}\mathscr{F}u_{1})*(\langle k\rangle^{1-s}\Lambda ^{s}\mathscr{F}u_{2})\right]
\right\|_{l_{k}^{\infty}L_{\tau}^{\infty}}\nonumber\\
&&\leq C\|\langle k\rangle^{-4s}\|_{l_{k}^{\infty}}
\prod_{j=1}^{2}\|u_{j}\|_{X_{1-s,s}}\leq C\prod_{j=1}^{2}\|u_{j}\|_{X_{1-s,s}}\leq C\prod_{j=1}^{2}\|u_{j}\|_{Z^{s}}.
\end{eqnarray*}
Case (c) can be proved similarly to case (b).

\noindent(4) Region $\Omega_{4}$.
In this  case, we consider (a)-(c) of Lemma 2.7, respectively.

\noindent
When (a) occurs: if $|\sigma|>4{\rm max}\left\{|\sigma_{1}|,|\sigma_{2}|\right\}$, we have
$\supp \left[\mathscr{F}u_{1}*\mathscr{F}u_{2}\right]\subset D_{2}.$
In this case, by using    Lemmas 2.3, 2.7,
since  $\frac{1}{6}+\epsilon\leq s\leq\frac{1}{2}-2\epsilon,$    we have that
\begin{eqnarray*}
&&\left\|\Lambda^{-1}\partial_{x}(1-\partial_{x}^{2})^{-1}
\prod_{j=1}^{2}(\partial_{x}u_{j})\right\|_{X^{s}}\leq
C\left\|\langle k\rangle^{-s}\langle \sigma \rangle^{s-\frac{2}{3}-\epsilon}\left[(|k|\mathscr{F}u_{1})*(|k|\mathscr{F}u_{2})\right]
\right\|_{l_{k}^{2}L_{\tau}^{2}}\nonumber\\
&&\leq C\left\|(J^{s}u_{1})(J^{s}u_{2})\right\|_{L_{xt}^{2}}
\leq C\|u_{1}\|_{X_{s,\frac{1}{6}+\epsilon}}\|u_{2}\|_{X_{s,\frac{1}{2}}}\leq C\prod_{j=1}^{2}\|u_{j}\|_{Z^{s}}.
\end{eqnarray*}
When $|\sigma|\leq4{\rm max}\left\{|\sigma_{1}|,|\sigma_{2}|\right\}$,
we have that $|\sigma|\sim |\sigma_{1}|$ or $|\sigma|\sim |\sigma_{2}|$.

\noindent
 When  $|\sigma|\sim |\sigma_{1}|$, if $\supp \left[\mathscr{F}u_{1}*\mathscr{F}u_{2}\right]\subset D_{1},$
then, $1\leq |k_{1}|\leq C,$ by using Lemma 2.3, since  $\frac{1}{6}+\epsilon\leq s\leq\frac{1}{2}-2\epsilon,$  we have that
\begin{eqnarray*}
&&\left\|\Lambda^{-1}(1-\partial_{x}^{2})^{-1}\partial_{x}\prod_{j=1}^{2}
(\partial_{x}u_{j})\right\|_{X^{s}}\leq C\left\|\langle k\rangle^{s}\langle \sigma \rangle^{-\frac{1}{6}-\epsilon}
\left[(|k|\mathscr{F}u_{1})*\mathscr{F}u_{2}\right]\right\|_{l_{k}^{2}L_{\tau}^{2}}
 \nonumber\\&&\leq\left\|(J^{s}u_{1})(J^{s}u_{2})\right\|_{L_{xt}^{2}}
\leq C\|u_{1}\|_{X_{s,\frac{1}{2}}}\|u_{2}\|_{X_{s,\frac{1}{6}+\epsilon}}\leq C\prod_{j=1}^{2}\|u_{j}\|_{Z^{s}};
\end{eqnarray*}
if $\supp \left[\mathscr{F}u_{1}*\mathscr{F}u_{2}\right]\subset D_{2},$
 by using Lemma 2.3, since $\frac{1}{6}+\epsilon\leq s\leq\frac{1}{2}-2\epsilon,$ we have
\begin{eqnarray*}
&&\left\|\Lambda^{-1}(1-\partial_{x}^{2})^{-1}\partial_{x}(\prod_{j=1}^{2}u_{j})
\right\|_{X^{s}}\leq C\left\|\langle k\rangle^{1-s}\langle \sigma \rangle^{s-\frac{2}{3}-\epsilon}
\left[(|k|\mathscr{F}u_{1})*\mathscr{F}u_{2}\right]\right\|_{l_{k}^{2}L_{\tau}^{2}}
\nonumber\\&&\leq\left\|(J^{s}u_{1})(J^{s}u_{2})\right\|_{L_{xt}^{2}}
\leq C\|u_{1}\|_{X_{s,\frac{1}{2}}}\|u_{2}\|_{X_{s,\frac{1}{6}+\epsilon}}\leq
 C\prod_{j=1}^{2}\|u_{j}\|_{Z^{s}};
\end{eqnarray*}
if $\supp \left[\mathscr{F}u_{1}*\mathscr{F}u_{2}\right]\subset D_{3},$
 by using Lemma 2.3, since $\frac{1}{6}+\epsilon\leq s\leq\frac{1}{2}-2\epsilon,$  we have that
\begin{eqnarray*}
&&\left\|\Lambda^{-1}\partial_{x}(1-\partial_{x}^{2})^{-1}\prod_{j=1}^{2}
(\partial_{x}u_{j})\right\|_{X^{s}}\leq C\left\|\langle k\rangle^{1-s}\langle \sigma \rangle^{s-1}
\left[(|k|\mathscr{F}u_{1})*\mathscr{F}u_{2}\right]\right\|_{l_{k}^{2}L_{\tau}^{2}} \nonumber\\&&\leq\left\|(J^{1-s}\Lambda^{s}u_{1})(J^{-s-1}u_{2})\right\|_{L_{xt}^{2}}\nonumber\\
&&\leq C\left\|(\langle k\rangle^{1-s}\langle\sigma\rangle^{s}\mathscr{F}u_{1})*(\langle k\rangle^{-1-s}\mathscr{F}u_{2})\right\|_{l_{k}^{2}L_{\tau}^{2}}\nonumber\\
&&\leq C\|u_{1}\|_{X_{1-s,s}}\|\langle k\rangle^{-1-s}\mathscr{F}u_{2}\|_{l_{k}^{1}L_{\tau}^{1}}\leq C\|u_{1}\|_{X_{1-s,s}}
\|u_{2}\|_{Y^{s}}\leq C\prod_{j=1}^{2}\|u_{j}\|_{Z^{s}}.
\end{eqnarray*}
When $|\sigma|\sim |\sigma_{2}|$,
if $\supp \left[\mathscr{F}u_{1}*\mathscr{F}u_{2}\right]\subset D_{1},$ then
 $1\leq |k_{1}|\leq C,$ by using Lemma 2.3, since $\frac{1}{6}+\epsilon\leq s\leq\frac{1}{2}-2\epsilon,$ we have that
\begin{eqnarray*}
&&\left\|\Lambda^{-1}(1-\partial_{x}^{2})^{-1}\partial_{x}\prod_{j=1}^{2}
(\partial_{x}u_{j})\right\|_{X^{s}}\leq C\left\|\langle k\rangle^{s}\langle \sigma \rangle^{-\frac{1}{6}-\epsilon}
\left[(|k|\mathscr{F}u_{1})*\mathscr{F}u_{2}\right]\right\|_{l_{k}^{2}L_{\tau}^{2}}
 \nonumber\\&&\leq\left\|(J^{s}u_{1})(J^{s}u_{2})\right\|_{L_{xt}^{2}}
\leq C\|u_{1}\|_{X_{s,\frac{1}{6}+\epsilon}}\|u_{2}\|_{X_{s,\frac{1}{2}}}\leq C\prod_{j=1}^{2}\|u_{j}\|_{Z^{s}};
\end{eqnarray*}
if $\supp \left[\mathscr{F}u_{1}*\mathscr{F}u_{2}\right]\subset D_{2},$
 by using Lemma 2.3, since $\frac{1}{6}+\epsilon\leq s\leq\frac{1}{2}-2\epsilon,$ we have that
\begin{eqnarray*}
&&\left\|\Lambda^{-1}(1-\partial_{x}^{2})^{-1}\partial_{x}(\prod_{j=1}^{2}u_{j})
\right\|_{X^{s}}\leq C\left\|\langle k\rangle^{1-s}
\langle \sigma \rangle^{s-\frac{2}{3}-\epsilon}
\left[(|k|\mathscr{F}u_{1})*\mathscr{F}u_{2}\right]\right\|_{l_{k}^{2}L_{\tau}^{2}} \nonumber\\&&\leq\left\|(J^{s}u_{1})(J^{1-s}\Lambda ^{s+\frac{1}{3}-\epsilon}u_{2})\right\|_{X_{0,-\frac{1}{6}-\epsilon}}
\leq C\|u_{1}\|_{X_{s,\frac{1}{2}}}\|u_{2}\|_{X_{1-s,s+\frac{1}{3}-\epsilon}}\leq C\prod_{j=1}^{2}\|u_{j}\|_{Z^{s}};
\end{eqnarray*}
if $\supp \left[\mathscr{F}u_{1}*\mathscr{F}u_{2}\right]\subset D_{3},$
 by using Lemma 2.3 and the Young inequality, since $\frac{1}{6}+\epsilon\leq s\leq\frac{1}{2}-2\epsilon,$ we have that
\begin{eqnarray*}
&&\left\|\Lambda^{-1}\partial_{x}(1-\partial_{x}^{2})^{-1}
\prod_{j=1}^{2}(\partial_{x}u_{j})\right\|_{X^{s}}\leq C\left\|\langle k\rangle^{1-s}\langle \sigma \rangle^{s-1}
\left[(|k|\mathscr{F}u_{1})*\mathscr{F}u_{2}\right]\right\|_{l_{k}^{2}L_{\tau}^{2}} \nonumber\\&&\leq\left\|(J^{-2}u_{1})(J^{1-s}\Lambda^{s}u_{2})\right\|_{L_{xt}^{2}}\nonumber\\&&\leq C\|J^{-2}u_{1}\|_{l_{k}^{1}L_{\tau}^{1}}\|u_{2}\|_{X_{1-s,s}}\leq
 C\|u_{1}\|_{Y^{s}}\|u_{2}\|_{X_{1-s,s}}\leq C\prod_{j=1}^{2}\|u_{j}\|_{Z^{s}}.
\end{eqnarray*}
(b): $|\sigma_{1}|={\rm max}\left\{|\sigma|,|\sigma_{1}|,|\sigma_{2}|\right\}.$
We consider $|\sigma_{1}|>4{\rm max}\left\{|\sigma|,|\sigma_{2}|\right\}$ and  $|\sigma_{1}|\leq4{\rm max}\left\{|\sigma|,|\sigma_{2}|\right\}$, respectively.

\noindent When $|\sigma_{1}|>4{\rm max}\left\{|\sigma|,|\sigma_{2}|\right\}$,
 then $\supp \mathscr{F}u_{1}\subset D_{3}$,  if $\supp\left[\mathscr{F}u_{1}*\mathscr{F}u_{2}\right]\subset D_{1},$
 by using Lemma 2.3,   since $\frac{1}{6}+\epsilon\leq s\leq\frac{1}{2}-2\epsilon,$  we have that
\begin{eqnarray*}
&&\left\|\Lambda^{-1}\partial_{x}(1-\partial_{x}^{2})^{-1}\prod_{j=1}^{2}
(\partial_{x}u_{j})\right\|_{X^{s}}\leq C\left\|\langle k\rangle^{s}\langle \sigma \rangle^{-\frac{1}{6}-\epsilon}
\left[(|k|\mathscr{F}u_{1})*\mathscr{F}u_{2}\right]\right\|_{l_{k}^{2}L_{\tau}^{2}}
\nonumber\\
&&\leq C\left\|(J^{1-s}\Lambda^{s}u_{1})(J^{-s}u_{2})\right\|_{X_{0,-\frac{1}{6}-\epsilon}}
\nonumber\\&&\leq C\|u_{1}\|_{X_{1-s,s}}\|u_{2}\|_{X_{s,\frac{1}{2}}}
\leq C\prod_{j=1}^{2}\|u_{j}\|_{Z^{s}};
\end{eqnarray*}
 if $\supp\left[\mathscr{F}u_{1}*\mathscr{F}u_{2}\right]\subset D_{2},$
  by using Lemma 2.3,   since $\frac{1}{6}+\epsilon\leq s\leq\frac{1}{2}-2\epsilon,$  we have that
\begin{eqnarray*}
&&\left\|\Lambda^{-1}\partial_{x}(1-\partial_{x}^{2})^{-1}\prod_{j=1}^{2}
(\partial_{x}u_{j})\right\|_{X^{s}}\leq C\left\|\langle k\rangle^{1-s}\langle \sigma \rangle^{s-\frac{2}{3}-\epsilon}
\left[(|k|\mathscr{F}u_{1})*\mathscr{F}u_{2}\right]\right\|_{l_{k}^{2}L_{\tau}^{2}}
\nonumber\\
&&\leq C\left\|\langle k\rangle^{s}\langle \sigma \rangle^{-\frac{1}{6}-\epsilon}
\left[(|k|\mathscr{F}u_{1})*\mathscr{F}u_{2}\right]\right\|_{l_{k}^{2}L_{\tau}^{2}}\nonumber\\
&&\leq C\left\|(J^{1-s}\Lambda^{s}u_{1})(J^{-s}u_{2})\right\|_{X_{0,-\frac{1}{6}-\epsilon}}
\nonumber\\&&\leq C\|u_{1}\|_{X_{1-s,s}}\|u_{2}\|_{X_{s,\frac{1}{2}}}
\leq C\prod_{j=1}^{2}\|u_{j}\|_{Z^{s}};
\end{eqnarray*}
 if $\supp\left[\mathscr{F}u_{1}*\mathscr{F}u_{2}\right]\subset D_{3},$
  by using Lemma 2.3 and the Young inequality,   since $\frac{1}{6}+\epsilon\leq s\leq\frac{1}{2}-2\epsilon,$  we have that
\begin{eqnarray*}
&&\left\|\Lambda^{-1}\partial_{x}(1-\partial_{x}^{2})^{-1}\prod_{j=1}^{2}
(\partial_{x}u_{j})\right\|_{X^{s}}\leq C\left\|\langle k\rangle^{1-s}\langle \sigma \rangle^{s-1}
\left[(|k|\mathscr{F}u_{1})*\mathscr{F}u_{2}\right]\right\|_{l_{k}^{2}L_{\tau}^{2}}
\nonumber\\
&&\leq C\left\|(J^{1-s}\Lambda^{s}u_{1})(J^{-2}u_{2})\right\|_{L_{xt}^{2}}
\nonumber\\&&\leq C\|u_{1}\|_{X_{1-s,s}}\|\langle k\rangle ^{-2}\mathscr{F}u_{2}\|_{l_{k}^{1}L_{\tau}^{1}}
\nonumber\\&&\leq C\|u_{1}\|_{X_{1-s,s}}\|u_{2}\|_{Y^{s}}\leq C\prod_{j=1}^{2}\|u_{j}\|_{Z^{s}}.
\end{eqnarray*}
When $|\sigma_{1}|\leq4{\rm max}\left\{|\sigma|,|\sigma_{2}|\right\}$
we have $|\sigma_{1}|\sim |\sigma|$ or $|\sigma_{1}|\sim |\sigma_{2}|$.

\noindent Case  $|\sigma_{1}|\sim |\sigma|$ can be proved similarly
 to case $|\sigma|={\rm max}\left\{|\sigma|,|\sigma_{1}|,|\sigma_{2}|\right\}.$

\noindent When $|\sigma_{1}|\sim |\sigma_{2}|$,
we have $\supp \mathscr{F}u_{1}\subset D_{3}.$

 \noindent
When
$\supp \left[\mathscr{F}u_{1}*\mathscr{F}u_{2}\right]\subset  D_{1},$ since $\frac{1}{6}+\epsilon\leq s\leq \frac{1}{2}-2\epsilon,$ by using Lemma 2.7, we have that
\begin{eqnarray*}
&&\left\|\Lambda^{-1}\partial_{x}(1-\partial_{x}^{2})^{-1}\prod_{j=1}^{2}
(\partial_{x}u_{j})\right\|_{X^{s}}\leq C\left\|\langle k\rangle ^{s}\langle \sigma\rangle ^{-\frac{1}{6}-\epsilon}
\left[(|k|\mathscr{F}u_{1})*\mathscr{F}u_{2}\right]\right\|_{l_{k}^{2}L_{\tau}^{2}}\nonumber\\
&&\leq C\left\|\langle k\rangle ^{\frac{1}{2}+\epsilon+s}\langle \sigma \rangle^{\frac{1}{3}}
\left[(|k|\mathscr{F}u_{1})*\mathscr{F}u_{2}\right]
\right\|_{l_{k}^{\infty}L_{\tau}^{\infty}}\nonumber\\
&&\leq C
\left\|\langle k\rangle ^{-s-\frac{1}{6}-\epsilon}\right\|_{l_{k}^{\infty}}
\|u_{1}\|_{X_{1-s,s}}\|u_{2}\|_{X_{1-s,s+\frac{1}{3}-\epsilon}}\leq C
\|u_{1}\|_{X_{1-s,s}}\|u_{2}\|_{X_{1-s,s+\frac{1}{3}-\epsilon}}\nonumber\\&&\leq C\prod_{j=1}^{2}\|u_{j}\|_{Z^{s}}.\label{3.02}
\end{eqnarray*}
When
$\supp \left[\mathscr{F}u_{1}*\mathscr{F}u_{2}\right]\subset  D_{2},$ since $\frac{1}{6}+\epsilon\leq s\leq -\frac{1}{2}-2\epsilon,$ by using Lemma 2.7, we have that
\begin{eqnarray*}
&&\left\|\Lambda^{-1}\partial_{x}(1-\partial_{x}^{2})^{-1}\prod_{j=1}^{2}
(\partial_{x}u_{j})\right\|_{X^{s}}\leq C\left\|\langle k\rangle ^{1-s}\langle \sigma\rangle ^{s-\frac{2}{3}-\epsilon}
\left[(|k|\mathscr{F}u_{1})*\mathscr{F}u_{2}\right]\right\|_{l_{k}^{2}L_{\tau}^{2}}\nonumber\\
&&\leq C\left\|\langle k\rangle ^{\frac{3}{2}-s+\epsilon}\langle \sigma \rangle^{s-\frac{1}{6}}
\left[(|k|\mathscr{F}u_{1})*\mathscr{F}u_{2}\right]
\right\|_{l_{k}^{\infty}L_{\tau}^{\infty}}\nonumber\\
&&\leq C
\left\|\langle k\rangle ^{-s-\frac{4}{3}-2\epsilon}\right\|_{l_{k}^{\infty}}
\|u_{1}\|_{X_{1-s,s}}\|u_{2}\|_{X_{1-s,s+\frac{1}{3}-\epsilon}}\leq C
\|u_{1}\|_{X_{1-s,s}}\|u_{2}\|_{X_{1-s,s+\frac{1}{3}-\epsilon}}\nonumber\\&&\leq C\prod_{j=1}^{2}\|u_{j}\|_{Z^{s}};\label{3.03}
\end{eqnarray*}
if $\supp \mathscr{F}u_{1}\subset D_{3}$ and $\supp
\left[\mathscr{F}u_{1}*\mathscr{F}u_{2}\right]\subset  D_{3},$ since $\frac{1}{6}+\epsilon\leq s\leq \frac{1}{2}-2\epsilon,$  we have that
\begin{eqnarray*}
&&\left\|\Lambda^{-1}\partial_{x}(1-\partial_{x}^{2})^{-1}\prod_{j=1}^{2}
(\partial_{x}u_{j})\right\|_{X^{s}}\leq C\left\|\langle k\rangle ^{1-s}\langle \sigma\rangle ^{s-1}
\left[(|k|\mathscr{F}u_{1})*\mathscr{F}u_{2}\right]\right\|_{l_{k}^{2}L_{\tau}^{2}}\nonumber\\
&&\leq C\left\|\langle k\rangle ^{\frac{3}{2}-s+\epsilon}\langle \sigma \rangle^{s-\frac{1}{2}+\epsilon}
\left[(|k|\mathscr{F}u_{1})*\mathscr{F}u_{2}\right]
\right\|_{l_{k}^{\infty}L_{\tau}^{\infty}}\nonumber\\
&&\leq C\left\|(|k|\mathscr{F}u_{1})
*\left(\langle k\rangle ^{2s+4\epsilon}\mathscr{F}u_{2}\right)\right\|_{l_{k}^{\infty}
l_{\tau}^{\infty}}\nonumber\\&&\leq C
\left\|\langle k\rangle ^{-2s-2+7\epsilon}\right\|_{l_{k}^{\infty}}
\|u_{1}\|_{X_{1-s,s}}\|u_{2}\|_{X_{1-s,s+\frac{1}{3}-\epsilon}}\leq C
\|u_{1}\|_{X_{1-s,s}}\|u_{2}\|_{X_{1-s,s+\frac{1}{3}-\epsilon}}\nonumber\\&&\leq C\prod_{j=1}^{2}\|u_{j}\|_{Z^{s}}.
\end{eqnarray*}
When case (c) occurs: we have that $\mathscr{F}u_{2}\subset D_{2}\cup D_{3}$.

\noindent
If $\mathscr{F}u_{2}\subset D_{2}$, we consider $\mathscr{F}u_{1}\subset D_{1}\cup D_{2}$ and $\mathscr{F}u_{1}\subset D_{3},$ respectively.

\noindent
When $\mathscr{F}u_{2}\subset D_{2}$ and $\mathscr{F}u_{1}\subset D_{1}\cup D_{2}$, by using Lemmas 2.6, 2.5, 2.3, since $\frac{1}{6}+\epsilon\leq s\leq\frac{1}{2}-2\epsilon$, we have that
\begin{eqnarray*}
&&\left\|\Lambda^{-1}\partial_{x}(1-\partial_{x}^{2})^{-1}\prod_{j=1}^{2}
(\partial_{x}u_{j})\right\|_{X^{s}}\leq C\left\|\langle k\rangle ^{s}\langle \sigma\rangle ^{-\frac{1}{6}-\epsilon}
\left[(|k|\mathscr{F}u_{1})*\mathscr{F}u_{2}\right]\right\|_{l_{k}^{2}L_{\tau}^{2}}\nonumber\\&&
\leq C\left\|\langle \sigma\rangle ^{-\frac{1}{6}-\epsilon}
\left[(\mathscr{F}u_{1})*\left(\langle k\rangle^{1-s}\langle\sigma\rangle^{s+\frac{1}{3}-\epsilon}\mathscr{F}u_{2}\right)\right]\right\|_{l_{k}^{2}L_{\tau}^{2}}
\nonumber\\
&&\leq C
\|u_{1}\|_{X_{s,\frac{1}{2}}}\|u_{2}\|_{X_{1-s,s+\frac{1}{3}-\epsilon}}\nonumber\\&&\leq C\prod_{j=1}^{2}\|u_{j}\|_{Z^{s}}.
\end{eqnarray*}
When $\mathscr{F}u_{2}\subset D_{2}$ and $\mathscr{F}u_{1}\subset  D_{3}$, by using Lemmas 2.6, 2.5, 2.3,  since $\frac{1}{6}+\epsilon\leq s\leq\frac{1}{2}-2\epsilon$,  we have that
\begin{eqnarray*}
&&\left\|\Lambda^{-1}\partial_{x}(1-\partial_{x}^{2})^{-1}\prod_{j=1}^{2}
(\partial_{x}u_{j})\right\|_{X^{s}}\leq C\left\|\langle k\rangle ^{s}\langle \sigma\rangle ^{-\frac{1}{6}-\epsilon}
\left[(|k|\mathscr{F}u_{1})*\mathscr{F}u_{2}\right]\right\|_{l_{k}^{2}L_{\tau}^{2}}\nonumber\\&&
\leq C\left\|\langle \sigma\rangle ^{-\frac{1}{6}-\epsilon}
\left[\left(\langle k\rangle^{2s}\langle\sigma\rangle^{-s-\frac{1}{3}-\epsilon}\mathscr{F}u_{1}\right)*\left(\langle k\rangle^{1-s}\langle\sigma\rangle^{s+\frac{1}{3}-\epsilon}\mathscr{F}u_{2}\right)\right]\right\|_{l_{k}^{2}L_{\tau}^{2}}
\nonumber\\&&\leq C
\|u_{1}\|_{X_{1-s,s}}\|u_{2}\|_{X_{1-s,s+\frac{1}{3}-\epsilon}}\nonumber\\
&&\leq C\prod_{j=1}^{2}\|u_{j}\|_{Z^{s}}.
\end{eqnarray*}
If $\mathscr{F}u_{2}\subset D_{3}$, we consider $\mathscr{F}u_{1}\subset D_{1}\cup D_{2}$ and $\mathscr{F}u_{1}\subset D_{3},$ respectively.

\noindent
When $\mathscr{F}u_{2}\subset D_{3}$ and $\mathscr{F}u_{1}\subset D_{1}\cup D_{2}$, by using Lemmas 2.6, 2.5, 2.3, since $\frac{1}{6}+\epsilon\leq s\leq\frac{1}{2}-2\epsilon$,  we have that
\begin{eqnarray*}
&&\left\|\Lambda^{-1}\partial_{x}(1-\partial_{x}^{2})^{-1}\prod_{j=1}^{2}
(\partial_{x}u_{j})\right\|_{X^{s}}\leq C\left\|\langle k\rangle ^{s}\langle \sigma\rangle ^{-\frac{1}{6}-\epsilon}
\left[(|k|\mathscr{F}u_{1})*\mathscr{F}u_{2}\right]\right\|_{l_{k}^{2}L_{\tau}^{2}}\nonumber\\&&
\leq C\left\|\langle \sigma\rangle ^{-\frac{1}{6}-\epsilon}
\left[(\mathscr{F}u_{1})*\left(\langle k\rangle^{1-s}\langle\sigma\rangle^{s}\mathscr{F}u_{2}\right)\right]\right\|_{l_{k}^{2}L_{\tau}^{2}}
\nonumber\\
&&\leq C
\|u_{1}\|_{X_{s,\frac{1}{2}}}\|u_{2}\|_{X_{1-s,s}}\nonumber\\&&\leq C\prod_{j=1}^{2}\|u_{j}\|_{Z^{s}}.
\end{eqnarray*}
When $\mathscr{F}u_{2}\subset D_{3}$ and $\mathscr{F}u_{1}\subset  D_{3}$, by using Lemmas 2.6, 2.5, 2.3, since $\frac{1}{6}+\epsilon\leq s\leq\frac{1}{2}-2\epsilon$, we have that
\begin{eqnarray*}
&&\left\|\Lambda^{-1}\partial_{x}(1-\partial_{x}^{2})^{-1}\prod_{j=1}^{2}
(\partial_{x}u_{j})\right\|_{X^{s}}\leq C\left\|\langle k\rangle ^{s}\langle \sigma\rangle ^{-\frac{1}{6}-\epsilon}
\left[(|k|\mathscr{F}u_{1})*\mathscr{F}u_{2}\right]\right\|_{l_{k}^{2}L_{\tau}^{2}}\nonumber\\&&
\leq C\left\|\langle \sigma\rangle ^{-\frac{1}{6}-\epsilon}
\left[\left(\langle k\rangle^{2s}\langle\sigma\rangle^{-s}\mathscr{F}u_{1}\right)*\left(\langle k\rangle^{1-s}\langle\sigma\rangle^{s}\mathscr{F}u_{2}\right)\right]\right\|_{l_{k}^{2}L_{\tau}^{2}}
\nonumber\\&&\leq C
\|u_{1}\|_{X_{1-s,s}}\|u_{2}\|_{X_{1-s,s+\frac{1}{3}-\epsilon}}\nonumber\\
&&\leq C\prod_{j=1}^{2}\|u_{j}\|_{Z^{s}}.
\end{eqnarray*}
(5) In region $\Omega_{5}$. In this region, we consider
$|k_{1}|\leq |k|^{-2}$ and $|k|^{-2}< |k_{1}|\leq 1,$ respectively.

\noindent
When $|k_{1}|\leq |k|^{-2}$, by using the Young inequality and Cauchy-Schwartz
inequality as well as Lemma 2.5, since $\frac{1}{6}+\epsilon\leq s\leq -\frac{1}{2}-2\epsilon,$  we have that
\begin{eqnarray*}
&&\left\|\Lambda^{-1}\partial_{x}(1-\partial_{x}^{2})^{-1}\prod_{j=1}^{2}
(\partial_{x}u_{j})\right\|_{X^{s}}\leq C\left\|
\langle k\rangle^{s}\langle \sigma\rangle ^{-\frac{1}{6}-\epsilon}
\left[(|k|\mathscr{F}u_{1})*\mathscr{F}u_{2}\right]\right\|_{l_{k}^{2}L_{\tau}^{2}}\nonumber\\
&&\leq C\left\|\langle k\rangle ^{s-2}\left[\mathscr{F}u_{1}*\mathscr{F}u_{2}\right]
\right\|_{l_{k}^{2}L_{\tau}^{2}}\nonumber\\
&&\leq C\left\|\left[\mathscr{F}u_{1}*(\langle k\rangle ^{s}\mathscr{F}u_{2})\right]
\right\|_{l_{k}^{\infty}L_{\tau}^{2}}\nonumber\\
&&\leq C\left\|\mathscr{F}u_{1}\right\|_{l_{k}^{2}l_{\tau}^{2}}\|u_{2}\|_{Y^{s}}
\leq C\prod_{j=1}^{2}\|u_{j}\|_{Z^{s}}.\label{3.04}
\end{eqnarray*}
When  $|k|^{-2}< |k_{1}|\leq 1.$
In this  region, we consider (a)-(c) of Lemma 2.7, respectively.

\noindent When (a) occurs:
 by using Lemma 2.5, the Young inequality and Cauchy-Schwartz inequality,since $\frac{1}{6}+\epsilon\leq s\leq\frac{1}{2}-2\epsilon,$  we have that
\begin{eqnarray*}
&&\left\|\Lambda^{-1}\partial_{x}(1-\partial_{x}^{2})^{-1}
\prod_{j=1}^{2}(\partial_{x}u_{j})\right\|_{X^{s}}
\leq C\left\|\langle k\rangle^{s}\langle \sigma\rangle ^{-\frac{1}{6}-\epsilon}
\left[(|k|\mathscr{F}u_{1})*\mathscr{F}u_{2}\right]\right\|_{l_{k}^{2}L_{\tau}^{2}}\nonumber\\
&&\leq C\left\|\left[(|k|^{\frac{5}{6}-\epsilon}\mathscr{F}u_{1})*(\langle k\rangle ^{s-\frac{1}{3}-2\epsilon}\mathscr{F}u_{2})\right]
\right\|_{l_{k}^{2}L_{\tau}^{2}}\nonumber\\
&&\leq C\left\||k|^{\frac{5}{6}-\epsilon}\mathscr{F}u_{1}\right\|_{l_{k}^{1}l_{\tau}^{2}}
\|u_{2}\|_{Y^{s}}\nonumber\\&&
\leq C\|\mathscr{F}u_{1}\|_{l_{k}^{2}L_{\tau}^{2}}\|u_{2}\|_{Y^{s}}
\leq C\prod_{j=1}^{2}\|u_{j}\|_{Z^{s}}.
\end{eqnarray*}
When (b) occurs:
 by using Lemma 2.5, the Young inequality and Cauchy-Schwartz inequality,  since $\frac{1}{6}+2\epsilon\leq s\leq\frac{1}{2}-2\epsilon,$ we have that
\begin{eqnarray*}
&&\left\|\Lambda^{-1}\partial_{x}(1-\partial_{x}^{2})^{-1}\prod_{j=1}^{2}
(\partial_{x}u_{j})\right\|_{X^{s}}\leq C\left\|\langle k\rangle ^{s}
\langle \sigma\rangle ^{-\frac{1}{6}-\epsilon}
\left[(|k|\mathscr{F}u_{1})*\mathscr{F}u_{2}\right]\right\|_{l_{k}^{2}L_{\tau}^{2}}\nonumber\\
&&\leq C\left\|\left[(|k|^{\frac{5}{6}-\epsilon}\langle \sigma \rangle
^{\frac{1}{6}+\epsilon}\mathscr{F}u_{1})*(\langle k\rangle ^{s-\frac{1}{3}-2\epsilon}\mathscr{F}u_{2})\right]
\right\|_{l_{k}^{2}L_{\tau}^{2}}\nonumber\\
&&\leq C\left\||k|^{\frac{5}{6}-\epsilon}\langle \sigma \rangle ^{\frac{1}{6}+\epsilon}\mathscr{F}u_{1}\right\|_{l_{k}^{1}l_{\tau}^{2}}\|u_{2}\|_{Y^{s}}\nonumber\\&&
\leq C\|\langle \sigma \rangle ^{\frac{1}{6}+\epsilon}\mathscr{F}u_{1}\|_{l_{k}^{2}L_{\tau}^{2}}\|u_{2}\|_{Y^{s}}\nonumber\\&&
\leq C\|u_{1}\|_{X_{0,\frac{1}{6}+\epsilon}}\|u_{2}\|_{Y^{s}}
\leq C\prod_{j=1}^{2}\|u_{j}\|_{Z^{s}}.
\end{eqnarray*}
When (c)  occurs:
by using Lemma 2.5, the Young inequality and Cauchy-Schwartz inequality, since $\frac{1}{6}+\epsilon\leq s\leq\frac{1}{2}-2\epsilon,$ we have that
\begin{eqnarray*}
&&\left\|\Lambda^{-1}\partial_{x}(1-\partial_{x}^{2})^{-1}\prod_{j=1}^{2}
(\partial_{x}u_{j})\right\|_{X^{s}}\leq C\left\|\langle k\rangle ^{s}\langle
\sigma\rangle ^{-\frac{1}{6}-\epsilon}
\left[\mathscr{F}u_{1}*\mathscr{F}u_{2}\right]\right\|_{l_{k}^{2}L_{\tau}^{2}}\nonumber\\
&&\leq C\left\|\left[(|k|^{\frac{5}{6}-\epsilon}\mathscr{F}u_{1})*(\langle k\rangle ^{s}
\langle \sigma \rangle ^{\frac{1}{6}+\epsilon}\mathscr{F}u_{2})\right]
\right\|_{l_{k}^{2}L_{\tau}^{2}}\nonumber\\
&&\leq C\left\||k|^{\frac{5}{6}-\epsilon}\mathscr{F}u_{1}\right\|_{l_{k}^{1}l_{\tau}^{1}}
\|u_{2}\|_{X_{s,\frac{1}{6}+\epsilon}}
\leq C\|u_{1}\|_{Y^{s}}\|u_{2}\|_{X_{s,\frac{1}{6}+\epsilon}}
\leq C\prod_{j=1}^{2}\|u_{j}\|_{Z^{s}}.
\end{eqnarray*}
(6)In region $\Omega_{6}$.
This region can be proved similarly to $\Omega_{4}$.

\noindent (7)In region $\Omega_{7}$.
This region can be proved similarly to $\Omega_{5}$.

\noindent (8)In region $\Omega_{8}$. In this  region, we consider (a)-(c) of Lemma 2.7, respectively.

\noindent When (a) occurs:
if $|\sigma|>4{\rm max}\left\{|\sigma_{1}|,|\sigma_{2}|\right\}$ and $\supp \mathscr{F}u_{1}\subset D_{1}\cup D_{2}$, by using Lemmas 2.5, 2.7, 2.3,
since $\supp\left( \mathscr{F}u_{1}*\mathscr{F}u_{2}\right) \subset D_{3}$ and $\frac{1}{6}+\epsilon\leq s\leq \frac{1}{2}-2\epsilon,$  we have that
\begin{eqnarray*}
&&\left\|\Lambda^{-1}\partial_{x}(1-\partial_{x}^{2})^{-1}\prod_{j=1}^{2}
(\partial_{x}u_{j})\right\|_{X^{s}}\leq C\left\|\langle k\rangle ^{2-s}\langle \sigma \rangle ^{s-1}
(\mathscr{F}u_{1}*\mathscr{F}u_{2})\right\|_{l_{k}^{2}L_{\tau}^{2}}\nonumber\\&&
\leq C\left\|(J^{s}u_{1})(J^{s}u_{2})\right\|_{L_{xt}^{2}}\leq
C\|u_{1}\|_{X_{s,1/2}}\|u_{2}\|_{X_{s,\frac{1}{6}+\epsilon}}\leq C\prod_{j=1}^{2}\|u_{j}\|_{Z^{s}}.
\end{eqnarray*}
If $|\sigma|>4{\rm max}\left\{|\sigma_{1}|,|\sigma_{2}|\right\}$ and $\supp \mathscr{F}u_{1}\subset  D_{3}$,
 by using Lemma 2.3, since $\frac{1}{6}+\epsilon\leq s\leq \frac{1}{2}-2\epsilon,$  we have
\begin{eqnarray*}
&&\left\|\Lambda^{-1}\partial_{x}(1-\partial_{x}^{2})^{-1}
\prod_{j=1}^{2}(\partial_{x}u_{j})\right\|_{X^{s}}\leq C\left\|\langle k\rangle ^{1+s}\langle \sigma \rangle ^{-\frac{1}{6}-\epsilon}
(\mathscr{F}u_{1}*\mathscr{F}u_{2})\right\|_{l_{k}^{2}L_{\tau}^{2}}\nonumber\\&&
\leq \left\|\left[J^{1-s}\Lambda^{s}u_{1}\right]\left[J^{-s}u_{2}\right]\right\|_{X_{0,-\frac{1}{6}-\epsilon}} \leq \|u_{1}\|_{X_{1-s,s}}\|u\|_{X_{s,\frac{1}{2}}}\leq C\prod_{j=1}^{2}\|u_{j}\|_{Z^{s}}.
\end{eqnarray*}
If $|\sigma|\leq 4{\rm max}\left\{|\sigma_{1}|,|\sigma_{2}|\right\},$ we have $|\sigma|\sim |\sigma_{1}|$ or $|\sigma|\sim |\sigma_{2}|.$

\noindent
When $|\sigma|\sim |\sigma_{1}|$. In this case, $\supp\left( \mathscr{F}u_{1}*\mathscr{F}u_{2}\right) \subset D_{3},$
 then by using the Young inequality, since $\frac{1}{6}+\epsilon\leq s\leq \frac{1}{2}-2\epsilon,$ we have that
\begin{eqnarray*}
&&\left\|\Lambda^{-1}\partial_{x}(1-\partial_{x}^{2})^{-1}\prod_{j=1}^{2}
(\partial_{x}u_{j})\right\|_{X^{s}}\leq C\left\|\langle k\rangle ^{2-s}
\langle \sigma \rangle^{s-1}\left(\mathscr{F}u_{1}*\mathscr{F}u_{2}\right) \right\|_{l_{k}^{2}L_{\tau}^{2}}\nonumber\\
&&\leq C\left\|(\langle k\rangle ^{1-s}
\langle \sigma \rangle^{s}\mathscr{F}u_{1})*\left[\langle k\rangle^{-2}\mathscr{F}u_{2}\right]\right\|_{l_{k}^{2}L_{\tau}^{2}}\nonumber\\
&&\leq C\|u_{1}\|_{X_{1-s,s}}\|\langle k\rangle^{-2}\mathscr{F}u_{2}\|_{l_{k}^{1}L_{\tau}^{1}}\nonumber\\
&&\leq C\|u_{1}\|_{X_{1-s,s}}\|\langle k\rangle^{s}\mathscr{F}u_{2}\|_{l_{k}^{2}L_{\tau}^{1}}\leq
C\prod_{j=1}^{2}\|u_{j}\|_{Z^{s}}.
\end{eqnarray*}
When  $|\sigma|\sim |\sigma_{2}|$, this case can be proved similarly to case  $|\sigma|\sim |\sigma_{1}|$.

\noindent When (b)  occurs: we consider case $|\sigma_{1}|>4{\rm max}\left\{|\sigma|,|\sigma_{2}|\right\}$ and $|\sigma_{1}|\leq4{\rm max}\left\{|\sigma|,|\sigma_{2}|\right\}$, respectively.

\noindent When $|\sigma_{1}|>4{\rm max}\left\{|\sigma|,|\sigma_{2}|\right\}$
which yields $\supp\mathscr{F}u_{1}\subset D_{3} $, we consider $\supp\mathscr{F}u_{2}\subset D_{1},$
$\supp\mathscr{F}u_{2}\subset D_{2}$ and $\supp\mathscr{F}u_{2}\subset D_{3},$ respectively.

\noindent If $\supp\mathscr{F}u_{2}\subset D_{1},$
  by using Lemmas 2.5, 2.3, since $\frac{1}{6}+\epsilon\leq s\leq \frac{1}{2}-2\epsilon,$ we have that
\begin{eqnarray*}
&&\left\|\Lambda^{-1}\partial_{x}(1-\partial_{x}^{2})^{-1}\prod_{j=1}^{2}
(\partial_{x}u_{j})\right\|_{X^{s}}\leq C\left\|\langle k\rangle ^{s+1}\langle \sigma \rangle ^{-\frac{1}{6}-\epsilon}(\mathscr{F}u_{1}*\mathscr{F}u_{2})\right\|_{l_{k}^{2}L_{\tau}^{2}}\nonumber\\
&&\leq C\left\|(J^{1-s}\Lambda ^{s}u_{1})(J^{-s}u_{2})\right\|_{X_{0,-\frac{1}{6}-\epsilon}}
\leq C\|u_{1}\|_{X_{1-s,s}}\|u_{2}\|_{X_{s,\frac{1}{2}}}\leq
C\prod_{j=1}^{2}\|u_{j}\|_{Z^{s}}.
\end{eqnarray*}
If $\supp\mathscr{F}u_{2}\subset D_{2},$
  by using Lemmas 2.5, 2.3,  since $\frac{1}{6}+\epsilon\leq s\leq \frac{1}{2}-2\epsilon,$ we have that
\begin{eqnarray*}
&&\left\|\Lambda^{-1}\partial_{x}(1-\partial_{x}^{2})^{-1}\prod_{j=1}^{2}
(\partial_{x}u_{j})\right\|_{X^{s}}\leq C\left\|\langle k\rangle ^{s+1}\langle \sigma \rangle ^{-\frac{1}{6}-\epsilon}(\mathscr{F}u_{1}*\mathscr{F}u_{2})\right\|_{l_{k}^{2}L_{\tau}^{2}}\nonumber\\
&&\leq C\left\|(J^{1-s}\Lambda ^{s}u_{1})(J^{-s}u_{2})\right\|_{X_{0,-\frac{1}{6}-\epsilon}}
\leq C\|u_{1}\|_{X_{1-s,s}}\|u_{2}\|_{X_{0,\frac{1}{2}}}\nonumber\\&&\leq
C\|u_{1}\|_{X_{1-s,s}}\|u_{2}\|_{X_{1-s,s+\frac{1}{3}-\epsilon}}\nonumber\\&&\leq
C\prod_{j=1}^{2}\|u_{j}\|_{Z^{s}}.
\end{eqnarray*}
If $\supp\mathscr{F}u_{2}\subset D_{3},$
  by using Lemmas 2.5, 2.3,  since $\frac{1}{6}+\epsilon\leq s\leq \frac{1}{2}-2\epsilon,$ we have that
\begin{eqnarray*}
&&\left\|\Lambda^{-1}\partial_{x}(1-\partial_{x}^{2})^{-1}
\prod_{j=1}^{2}(\partial_{x}u_{j})\right\|_{X^{s}}\leq C\left\|\langle k\rangle ^{s+1}
\langle \sigma\rangle ^{-\frac{1}{6}-\epsilon}
\left[\mathscr{F}u_{1}*\mathscr{F}u_{2}\right]\right\|_{l_{k}^{2}L_{\tau}^{2}}\nonumber\\
&&\leq C\left\|\langle k\rangle ^{s+\frac{3}{2}+\epsilon}\langle \sigma \rangle^{\frac{1}{3}}
\left[\mathscr{F}u_{1}*\mathscr{F}u_{2}\right]
\right\|_{l_{k}^{\infty}L_{\tau}^{\infty}}\nonumber\\
&&\leq C\left\|\left(\langle k\rangle ^{s+\frac{5}{2}+\epsilon}\mathscr{F}u_{1}
\right)*\mathscr{F}u_{2}\right\|_{l_{k}^{\infty}l_{\tau}^{\infty}}\nonumber\\&&\leq C
\left\|\langle k\rangle ^{-3s+\frac{1}{2}+\epsilon}\right\|_{l_{k}^{\infty}}
\prod_{j=1}^{2}\|u_{j}\|_{X_{1-s,s}}\leq C
\prod_{j=1}^{2}\|u_{j}\|_{X_{1-s,s}}\leq C\prod_{j=1}^{2}\|u_{j}\|_{Z^{s}}.
\end{eqnarray*}
If $|\sigma_{1}|\leq 4{\rm max}\left\{|\sigma|,|\sigma_{2}|\right\},$ we have that
$|\sigma_{1}|\sim |\sigma|$ or $|\sigma_{1}|\sim |\sigma_{2}|.$

\noindent
When $|\sigma_{1}|\sim |\sigma|$, this case can be proved similarly to
$|\sigma|={\rm max}\left\{|\sigma|,|\sigma_{1}|,|\sigma_{2}|\right\}.$

\noindent
If $|\sigma_{1}|\sim |\sigma_{2}|,$
 $\supp\mathscr{F}u_{j}\subset D_{3}$ with $j=1,2$,  since $\frac{1}{6}+\epsilon\leq s\leq\frac{1}{2}-2\epsilon,$ we have that
\begin{eqnarray*}
&&\left\|\Lambda^{-1}\partial_{x}(1-\partial_{x}^{2})^{-1}
\prod_{j=1}^{2}(\partial_{x}u_{j})\right\|_{X^{s}}\leq C\left\|\langle k\rangle ^{s+1}
\langle \sigma\rangle ^{-\frac{1}{6}-\epsilon}
\left[\mathscr{F}u_{1}*\mathscr{F}u_{2}\right]\right\|_{l_{k}^{2}L_{\tau}^{2}}\nonumber\\
&&\leq C\left\|\langle k\rangle ^{s+\frac{3}{2}+\epsilon}\langle \sigma \rangle^{\frac{1}{3}}
\left[\mathscr{F}u_{1}*\mathscr{F}u_{2}\right]
\right\|_{l_{k}^{\infty}L_{\tau}^{\infty}}\nonumber\\
&&\leq C\left\|\left(\langle k\rangle ^{s+\frac{5}{2}+\epsilon}\mathscr{F}u_{1}
\right)*\mathscr{F}u_{2}\right\|_{l_{k}^{\infty}l_{\tau}^{\infty}}\nonumber\\&&\leq C
\left\|\langle k\rangle ^{-3s+\frac{1}{2}+\epsilon}\right\|_{l_{k}^{\infty}}
\prod_{j=1}^{2}\|u_{j}\|_{X_{1-s,s}}\leq C
\prod_{j=1}^{2}\|u_{j}\|_{X_{1-s,s}}\leq C\prod_{j=1}^{2}\|u_{j}\|_{Z^{s}}.
\end{eqnarray*}
$(c): |\sigma_{2}|={\rm max}\left\{|\sigma|,|\sigma_{1}|,|\sigma_{2}|\right\}$ can be proved similarly to
$ |\sigma_{1}|={\rm max}\left\{|\sigma|,|\sigma_{1}|,|\sigma_{2}|\right\}$.

We have completed the proof of Lemma 3.1.

\noindent {\bf Remark 3.}  In the process of proving Lemma 3.1, the cases $ \Omega_{3}$ is  the most difficult to deal with and cases $ \Omega_{j}$ with $j=2,3$
require restriction $s>\frac{1}{6},$  more precisely,  case (b) of region $\Omega_{2}$ determines that $s\geq \frac{1}{6}+\epsilon$ is necessary.

\begin{Lemma}\label{Lemma3.2}
Let  $\frac{1}{6}+\epsilon\leq s\leq\frac{1}{2}-2\epsilon$, $0<\epsilon\ll1$. Then, we have that
\begin{eqnarray}
      \left\|\Lambda^{-1}\partial_{x}(1-\partial_{x}^{2})^{-1}
      \prod_{j=1}^{2}(\partial_{x}u_{j})\right\|_{Y^{s}}
      \leq C\prod\limits_{j=1}^{2}\|u_{j}\|_{Z^{s}}.
        \label{3.05}
\end{eqnarray}
\end{Lemma}
{\bf Proof.}
 Obviously, $\left(\R\times\dot{Z}_{\lambda}\right)^{2}\subset
 \bigcup\limits_{j=1}^{8}\Omega_{j},$
where
$\Omega_{j} (1\leq j\leq 8)$ are defined as Lemma 3.1.

\noindent
(1) In region $\Omega_{1}$.
By using Lemma 2.5 and the H\"older inequality
as well as the Cauchy-Schwartz inequality, since
$\frac{1}{6}+\epsilon\leq s\leq\frac{1}{2}-2\epsilon,$ we have that
\begin{eqnarray*}
&&\left\|\Lambda^{-1}\partial_{x}(1-\partial_{x}^{2})^{-1}
\prod_{j=1}^{2}(\partial_{x}u_{j})\right\|_{Y^{s}}\leq C
\left\|\Lambda^{-1}\partial_{x}(1-\partial_{x}^{2})^{-1}
\prod_{j=1}^{2}(\partial_{x}u_{j})\right\|_{X_{s,\frac{5}{6}-\epsilon}}\nonumber\\
&&\leq C\left\||k|\langle \sigma\rangle^{-\frac{1}{6}-\epsilon}
\left(\mathscr{F}u_{1}*\mathscr{F}u_{2}\right)
\right\|_{l_{k}^{2}L_{\tau}^{2}}\nonumber\\
&&\leq C\||k|\|_{l_{k}^{2}}\left\|\mathscr{F}u_{1}*\mathscr{F}u_{2}\right\|_{l_{k}^{\infty}L_{\tau}^{2}}\nonumber\\
&&\leq C\|\mathscr{F}u_{1}\|_{l_{k}^{2}L_{\tau}^{2}}\|\mathscr{F}u_{2}\|_{l_{k}^{2}L_{\tau}^{1}}
\leq C\|u_{1}\|_{X_{s,\frac{1}{6}+\epsilon}}\|u_{2}\|_{Y^{s}}\leq C
\prod_{j=1}^{2}\|u_{j}\|_{Z^{s}}.
\end{eqnarray*}
(2) In region $\Omega_{2}$.
We consider (a)-(c) of Lemma 2.7, respectively.

\noindent (a) Case $|\sigma|={\rm max}\left\{|\sigma|,|\sigma_{1}|,|\sigma_{2}|\right\},$
by using Lemma 2.7,  the Young inequality,
   since $\frac{1}{6}+\epsilon\leq s\leq\frac{1}{2}-2\epsilon,$ we have that
\begin{eqnarray*}
&&\left\|\Lambda^{-1}\partial_{x}(1-\partial_{x}^{2})^{-1}\prod_{j=1}^{2}
(\partial_{x}u_{j})\right\|_{Y^{s}}\leq\left\|\langle k\rangle^{s-1}\langle\sigma\rangle^{-1}\left[(|k|\mathscr{F}u_{1})*(
|k|\mathscr{F}u_{2})\right]\right\|_{l_{k}^{2}L_{\tau}^{1}}\nonumber\\
&&\leq C\left\|(\langle k\rangle ^{s}\mathscr{F}u_{1})*(\langle k\rangle ^{s}\mathscr{F}u_{2})\right\|_{l_{k}^{\infty}L_{\tau}^{1}}\nonumber\\
&&\leq \left\|(\langle k\rangle ^{s}\mathscr{F}u_{1})*(\langle k\rangle ^{s}\mathscr{F}u_{2})\right\|_{l_{k}^{\infty}L_{\tau}^{1}}\leq C\prod_{j=1}^{2}\|u_{j}\|_{Y^{s}}\leq C\prod_{j=1}^{2}\|u_{j}\|_{Z^{s}}.
\end{eqnarray*}
(b) Case $|\sigma_{1}|={\rm max}\left\{|\sigma|,|\sigma_{1}|,|\sigma_{2}|\right\},$
in this case, we consider the following cases:
\begin{eqnarray*}
(i): |\sigma_{1}|>4{\rm max}\left\{|\sigma|,|\sigma_{2}|\right\},\quad
(ii):|\sigma_{1}|\leq4{\rm max}\left\{|\sigma|,|\sigma_{2}|\right\},
\end{eqnarray*}
respectively.

\noindent
When (i) occurs:
if $\supp \mathscr{F}u_{1}\subset D_{1}$ which yields that $1\leq|k|\leq C$,
by using Lemmas 2.5, 2.7, 2.3, since
$\frac{1}{6}+\epsilon\leq s\leq\frac{1}{2}-2\epsilon,$
 we have that
\begin{eqnarray*}
&&\left\|\Lambda^{-1}\partial_{x}(1-\partial_{x}^{2})^{-1}
\prod_{j=1}^{2}(\partial_{x}u_{j})\right\|_{Y^{s}}\leq C\left\|\Lambda^{-1}\partial_{x}(1-\partial_{x}^{2})^{-1}
\prod_{j=1}^{2}(\partial_{x}u_{j})\right\|_{Z^{s}}\nonumber\\&&
\leq C\left\|\partial_{x}(1-\partial_{x}^{2})^{-1}\prod_{j=1}^{2}
(\partial_{x}u_{j})\right\|_{X_{s,-\frac{1}{6}-\epsilon}}
\leq C\left\|\langle k\rangle ^{s-1}\langle \sigma \rangle ^{-\frac{1}{6}-\epsilon}
(|k|\mathscr{F}u_{1})*(|k|\mathscr{F}u_{2})\right\|_{l_{k}^{2}L_{\tau}^{2}}\nonumber\\&&
\leq C\left\|\langle \sigma \rangle ^{-\frac{1}{6}-\epsilon}(\langle k\rangle^{s}
\langle\sigma\rangle^{\frac{5}{6}-\epsilon}\mathscr{F}u_{1})*(\langle k\rangle ^{-s+\frac{1}{3}+2\epsilon}\mathscr{F}u_{2})\right\|_{l_{k}^{2}L_{\tau}^{2}}\nonumber\\
&&\leq C\left\|\left(J^{s}\Lambda ^{\frac{5}{6}-\epsilon}u_{1}\right)
\left(J^{-s+\frac{1}{3}+2\epsilon}u_{2}\right)\right\|_{X_{0,-\frac{1}{6}-\epsilon}}\leq C\|u_{1}\|_{X_{s,\frac{5}{6}-\epsilon}}\|u_{2}\|_{X_{s,\frac{1}{2}}}\leq
C\prod_{j=1}^{2}\|u_{j}\|_{Z^{s}};
\end{eqnarray*}
if $\supp \mathscr{F}u_{1}\subset D_{2},$  by using Lemmas 2.5, 2.7, 2.3, since $\frac{1}{6}+\epsilon\leq s\leq\frac{1}{2}-2\epsilon$,
 we have that
\begin{eqnarray*}
&&\left\|\Lambda^{-1}\partial_{x}(1-\partial_{x}^{2})^{-1}
\prod_{j=1}^{2}(\partial_{x}u_{j})\right\|_{Y^{s}}\leq C\left\|\Lambda^{-1}\partial_{x}(1-\partial_{x}^{2})^{-1}
\prod_{j=1}^{2}(\partial_{x}u_{j})\right\|_{Z^{s}}\nonumber\\&&
\leq C\left\|\partial_{x}(1-\partial_{x}^{2})^{-1}
\prod_{j=1}^{2}(\partial_{x}u_{j})\right\|_{X_{s,-\frac{1}{6}-\epsilon}}
\nonumber\\&&\leq \left\|\langle k\rangle ^{s-1}\langle \sigma \rangle ^{-\frac{1}{6}-\epsilon}
\left[(|k|\mathscr{F}u_{1})*(|k|\mathscr{F}u_{2})\right]
\right\|_{l_{k}^{2}L_{\tau}^{2}}\nonumber\\&&
\leq C\left\|\left(J^{1-s}\Lambda ^{s+\frac{1}{3}-\epsilon}u_{1}\right)
\left(J^{-s+\frac{1}{3}+2\epsilon}u_{2}\right)
\right\|_{X_{0,-\frac{1}{6}-\epsilon}}\nonumber\\
&&\leq C\|u\|_{X_{1-s,s+\frac{1}{3}-\epsilon}}\|u\|_{X_{s,\frac{1}{2}}}
\leq
C\prod_{j=1}^{2}\|u_{j}\|_{Z^{s}}.
\end{eqnarray*}
When (ii) occurs: we have $|\sigma_{1}|\sim |\sigma|$ or $|\sigma_{1}|\sim |\sigma_{2}|$.

\noindent
When $|\sigma_{1}|\sim |\sigma|$  is valid, this case can be proved similarly to
$|\sigma|={\rm max}\left\{|\sigma|,|\sigma_{1}|,|\sigma_{2}|\right\}.$
When $|\sigma_{1}|\sim |\sigma_{2}|$, if  $\supp \mathscr{F}u_{1}\subset D_{1}$ which yields that $1\leq |k|\leq C$,
by using Lemma 2.5, 2.7, 2.3, since $\frac{1}{6}+\epsilon\leq s\leq\frac{1}{2}-2\epsilon,$
 we have that
\begin{eqnarray*}
&&\left\|\Lambda^{-1}\partial_{x}(1-\partial_{x}^{2})^{-1}\prod_{j=1}^{2}(\partial_{x}u_{j})\right\|
_{Y^{s}}\leq C\left\|\Lambda^{-1}\partial_{x}(1-\partial_{x}^{2})^{-1}\prod_{j=1}^{2}(\partial_{x}u_{j})
\right\|_{Z^{s}}\nonumber\\&&
\leq C\left\|\partial_{x}(1-\partial_{x}^{2})^{-1}\prod_{j=1}^{2}(\partial_{x}u_{j})\right\|
_{X_{s,-\frac{1}{6}-\epsilon}}\nonumber\\&&\leq \left\|\langle k\rangle ^{s-1}\langle \sigma \rangle ^{-\frac{1}{6}-\epsilon}
\left[(|k|\mathscr{F}u_{1})*(|k|\mathscr{F}u_{2})\right]\right\|_{l_{k}^{2}L_{\tau}^{2}}\nonumber\\
&&\leq C\left\|\left(J^{s}\Lambda ^{\frac{5}{6}-\epsilon}u_{1}\right)
\left(J^{-s+\frac{1}{3}+2\epsilon}u_{2}\right)\right\|_{X_{0,-\frac{1}{6}-\epsilon}}\nonumber\\&&\leq  C
\|u_{1}\|_{X_{s,\frac{5}{6}-\epsilon}}\|u_{2}\|_{X_{s,\frac{1}{2}}}\leq
C\prod_{j=1}^{2}\|u_{j}\|_{Z^{s}};
\end{eqnarray*}
if $\supp u_{1}\subset D_{2},$ we  can assume that $\supp u_{2}\subset D_{2}$,
we can assume that  $|\sigma|\leq C|k_{1}|^{3},$
 by using the H\"older  inequality  and the  Young inequality,
  since $\frac{1}{6}+\epsilon\leq s\leq\frac{1}{2}-2\epsilon$,
 we  have that
\begin{eqnarray*}
&&\left\|\Lambda^{-1}\partial_{x}(1-\partial_{x}^{2})^{-1}
\prod_{j=1}^{2}(\partial_{x}u_{j})\right\|_{Y^{s}}\leq C\left\|\Lambda^{-1}\partial_{x}(1-\partial_{x}^{2})^{-1}
\prod_{j=1}^{2}(\partial_{x}u_{j})\right\|_{Z^{s}}\nonumber\\&&
\leq C\left\|\partial_{x}(1-\partial_{x}^{2})^{-1}\prod_{j=1}^{2}
(\partial_{x}u_{j})\right\|_{X_{s,-\frac{1}{6}-\epsilon}}
\nonumber\\&&\leq \left\|\langle k\rangle ^{s-1}
\langle \sigma \rangle ^{-\frac{1}{6}-\epsilon}
\left[(|k|\mathscr{F}u_{1})*(|k|\mathscr{F}u_{2})\right]
\right\|_{l_{k}^{2}L_{\tau}^{2}}\nonumber\\
&&\leq C\left\|\left(J^{1-s}\Lambda ^{s+\frac{1}{3}-\epsilon}u_{1}\right)
\left(J^{-s+\frac{1}{3}+2\epsilon}u_{2}\right)
\right\|_{X_{0,-\frac{1}{6}-\epsilon}}\nonumber\\
&&\leq C
\|u_{1}\|_{X_{1-s,s+\frac{1}{3}+\epsilon}}\|u_{2}\|_{X_{s,\frac{1}{2}}}\leq C\prod_{j=1}^{2}\|u_{j}\|_{Z^{s}}.
\end{eqnarray*}
(c) Case $|\sigma_{2}|=\left\{|\sigma|,|\sigma_{1}|,|\sigma_{2}|\right\}.$
 This case can be proved similarly to case (b).

\noindent (3) Region $\Omega_{3}$.
We  consider $|k|\leq |k_{1}|^{-2}$  and  $|k_{1}|^{-2}<|k|\leq 1,$
respectively.

\noindent
When $|k|\leq |k_{1}|^{-2}$, by using Lemma 2.5 and the Young inequality,
since $\frac{1}{6}+\epsilon\leq s\leq\frac{1}{2}-2\epsilon$, we have that
\begin{eqnarray*}
&&\left\|\Lambda^{-1}\partial_{x}(1-\partial_{x}^{2})^{-1}
\prod_{j=1}^{2}(\partial_{x}u_{j})\right\|_{Y^{s}}\leq C
\left\||k|\langle \sigma\rangle ^{-\frac{1}{6}-\epsilon}
\left[(|k|\mathscr{F}u_{1})*(|k|\mathscr{F}u_{2})\right]\right\|_{l_{k}^{2}L_{\tau}^{2}}\nonumber\\
&&\leq C\left\|\left[\mathscr{F}u_{1}*
\mathscr{F}u_{2}\right]
\right\|_{l_{k}^{\infty}L_{\tau}^{2}}\leq C\|u_{1}\|_{X_{0,0}}\|u_{2}\|_{Y^{0}}\leq C\prod_{j=1}^{2}\|u_{j}\|_{Z^{s}}.
\end{eqnarray*}
When  $|k_{1}|^{-2}<|k|\leq 1,$
 we consider (a)-(c) of Lemma 2.7, respectively.

\noindent When (a) occurs:  by using the H\"older inequality and the Young
inequality, since $\frac{1}{6}+\epsilon\leq s\leq\frac{1}{2}-2\epsilon$, we have that
\begin{eqnarray*}
&&\left\|\Lambda^{-1}\partial_{x}(1-\partial_{x}^{2})^{-1}
\prod_{j=1}^{2}(\partial_{x}u_{j})\right\|_{Y^{s}}\leq
 C\left\||k|\langle k\rangle^{s}\langle \sigma \rangle^{-1}\left[(|k|\mathscr{F}u_{1})*(|k|\mathscr{F}u_{2})\right]
\right\|_{l_{k}^{2}L_{\tau}^{1}}\nonumber\\
&&\leq C\left\|\left[\mathscr{F}u_{1}*\mathscr{F}u_{2}\right]
\right\|_{l_{k}^{\infty}L_{\tau}^{1}}\nonumber\\
&&\leq C\prod_{j=1}^{2}\|\mathscr{F}u_{j}\|_{l_{k}^{2}L_{\tau}^{1}}\leq C\prod_{j=1}^{2}\|u_{j}\|_{Y^{s}}
\leq C\prod_{j=1}^{2}\|u_{j}\|_{Z^{s}}.
\end{eqnarray*}
When  (b) occurs:  we consider  $|\sigma_{1}|>4{\rm max}\left\{|\sigma|,|\sigma_{2}|\right\}$  and
$|\sigma_{1}|\leq4{\rm max}\left\{|\sigma|,|\sigma_{2}|\right\}$, respectively.

\noindent
When $|\sigma_{1}|>4{\rm max}\left\{|\sigma|,|\sigma_{2}|\right\}$,
in this case, $\supp \mathscr{F}u_{1}\subset D_{1}$, by using
$X_{s,\frac{1}{2}+\epsilon}\hookrightarrow Y^{s}$
and the H\"older inequality and the Young inequality, since
$\frac{1}{6}+\epsilon\leq s\leq\frac{1}{2}-2\epsilon,$
 we have that
\begin{eqnarray*}
&&\left\|\Lambda^{-1}\partial_{x}(1-\partial_{x}^{2})^{-1}
\prod_{j=1}^{2}(\partial_{x}u_{j})\right\|_{Y^{s}}
\leq C\left\||k|\langle \sigma \rangle ^{-\frac{1}{2}+\epsilon}\left[(|k|\mathscr{F}u_{1})*(|k|\mathscr{F}u_{2})\right]
\right\|_{l_{k}^{2}L_{\tau}^{2}}\nonumber\\
&&\leq C\left\||k|^{\frac{1}{6}+\epsilon}\left(\langle k\rangle^{s}
\langle \sigma
\rangle ^{\frac{5}{6}-\epsilon}\mathscr{F}u_{1}\right)*\left
(\langle k\rangle ^{-s+\frac{1}{3}+2\epsilon}
\mathscr{F}u_{2}\right)\right\|_{l_{k}^{2}L_{\tau}^{2}}\nonumber\\
&&\leq C\left\|\left(\langle k\rangle^{s}\langle \sigma \rangle^{\frac{5}{6}-\epsilon}\mathscr{F}u_{1}\right)*\left(\langle k\rangle ^{-s+\frac{1}{3}+2\epsilon}\mathscr{F}u_{2}\right)\right\|_{
l_{k}^{\infty}L_{\tau}^{2}}\nonumber\\
&&\leq C\|u_{1}\|_{X_{s,\frac{5}{6}-\epsilon}}\|u_{2}\|_{Y^{s}}.
\end{eqnarray*}
When $|\sigma_{1}|\leq4{\rm max}\left\{|\sigma|,|\sigma_{2}|\right\}$, we have
that $|\sigma_{1}|\sim |\sigma|$ or $|\sigma_{1}|\sim |\sigma_{2}|.$

\noindent
When $|\sigma_{1}|\sim |\sigma|$, this case can be proved similarly to case
$|\sigma|={\rm max}\left\{|\sigma|,|\sigma_{1}|,|\sigma_{2}|\right\}.$
When $|\sigma_{1}|\sim |\sigma_{2}|$, if $\supp \mathscr{F}u_{j}\subset D_{1}$ with $j=1,2$, by using $X_{s,\frac{1}{2}+\epsilon}\hookrightarrow Y^{s}$ and the Young inequality, since $\frac{1}{6}+\epsilon\leq s\leq\frac{1}{2}-2\epsilon,$  we have that
\begin{eqnarray*}
&&\left\|\Lambda^{-1}\partial_{x}(1-\partial_{x}^{2})^{-1}
\prod_{j=1}^{2}(\partial_{x}u_{j})\right\|_{Y^{s}}\leq
 C\left\||k|\langle \sigma \rangle ^{-\frac{1}{2}+\epsilon}
 \left[(|k|\mathscr{F}u_{1})*(|k|\mathscr{F}u_{2})\right]
 \right\|_{l_{k}^{2}L_{\tau}^{2}}\nonumber\\
&&\leq C\left\||k|^{\frac{1}{6}+\epsilon}
\left(\langle k\rangle^{s}\langle \sigma
\rangle ^{\frac{5}{6}-\epsilon}\mathscr{F}u_{1}\right)
*\left(\langle k\rangle ^{-s+\frac{1}{3}+2\epsilon}
\mathscr{F}u_{2}\right)\right\|_{l_{k}^{2}L_{\tau}^{2}}\nonumber\\
&&\leq C\left\|\left(\langle k\rangle^{s}\langle \sigma \rangle^{\frac{5}{6}-\epsilon}\mathscr{F}u_{1}\right)
*\left(\langle k\rangle ^{-s+\frac{1}{3}+2\epsilon}\mathscr{F}u_{2}\right)
\right\|_{l_{k}^{\infty}L_{\tau}^{2}}\nonumber\\
&&\leq C\|u_{1}\|_{X_{s,\frac{5}{6}-\epsilon}}\|u_{2}\|_{Y^{s}}\nonumber\\
&&\leq C\prod_{j=1}^{2}\|u_{j}\|_{Z^{s}};
\end{eqnarray*}
if $\supp \mathscr{F}u_{j}\subset D_{2}$ with $j=1,2$, by using $X_{s,\frac{1}{2}+\epsilon}\hookrightarrow Y^{s}$,
Lemmas 2.7, 2.5, 2.3, since $\frac{1}{6}+\epsilon\leq s\leq\frac{1}{2}-2\epsilon$, we have that
\begin{eqnarray*}
&&\left\|\Lambda^{-1}\partial_{x}(1-\partial_{x}^{2})^{-1}
\prod_{j=1}^{2}(\partial_{x}u_{j})\right\|_{Y^{s}}\leq
 C\left\||k|\langle \sigma \rangle ^{-\frac{1}{2}+\epsilon}
 \left[(|k|\mathscr{F}u_{1})*(|k|\mathscr{F}u_{2})\right]
 \right\|_{l_{k}^{2}L_{\tau}^{2}}\nonumber\\
&&\leq C\left\|\left(\langle k\rangle^{1-s}\langle \sigma
\rangle ^{s+\frac{1}{3}-\epsilon}
\mathscr{F}u_{1}\right)*\left(\langle k\rangle ^{-s+\frac{1}{3}+2\epsilon}
\mathscr{F}u_{2}\right)
\right\|_{X^{0,-\frac{1}{2}+\epsilon}}\nonumber\\
&&\leq C\|u_{1}\|_{X_{1-s,s+\frac{1}{3}-\epsilon}}
\|u_{2}\|_{X_{s,\frac{1}{6}+\epsilon}}\nonumber\\
&&\leq C\|u_{1}\|_{X_{1-s,s+\frac{1}{3}-\epsilon}}
\|u_{2}\|_{X_{s,\frac{1}{6}+\epsilon}}\leq C\prod_{j=1}^{2}\|u_{j}\|_{Z^{s}};
\end{eqnarray*}
if $\supp \mathscr{F}u_{j}\subset D_{3}$ with $j=1,2$,  by using $X_{s,\frac{1}{2}+\epsilon}\hookrightarrow Y^{s}$ and Lemma 2.5,
since $\frac{1}{6}+\epsilon\leq s\leq \frac{1}{2}-2\epsilon,$ we have that
\begin{eqnarray*}
&&\left\|\Lambda^{-1}\partial_{x}(1-\partial_{x}^{2})^{-1}
\prod_{j=1}^{2}(\partial_{x}u_{j})\right\|_{Y^{s}}\leq
 C\left\|\Lambda^{-1}\partial_{x}(1-\partial_{x}^{2})^{-1}
 \prod_{j=1}^{2}(\partial_{x}u_{j})
 \right\|_{Z^{s}}\nonumber\\
&&\leq C\left\||k|\langle k\rangle ^{s-2}\langle
\sigma \rangle ^{-\frac{1}{2}+\epsilon}\left(|k|
\mathscr{F}u_{1}\right)*\left(|k|\mathscr{F}u_{2}\right)
\right\|_{l_{k}^{2}L_{\tau}^{2}}\nonumber\\
&&\leq C\left\|\left(\langle k\rangle^{1-s}\langle \sigma
\rangle ^{s}
\mathscr{F}u_{1}\right)*\left(\langle k\rangle ^{-2s+1}
\mathscr{F}u_{2}\right)
\right\|_{X_{0,-\frac{1}{2}+\epsilon}}\nonumber\\
&&\leq C\left\|u_{1}\right\|_{X_{1-s,s}}\left\|u_{2}
\right\|_{X_{1-2s,\frac{1}{6}+\epsilon}}\nonumber\\
&&\leq C\prod_{j=1}^{2}\left\|u_{j}\right\|_{X_{1-s,s}}
\leq C
\prod_{j=1}^{2}\|u_{j}\|_{Z^{s}}.
\end{eqnarray*}
When case (c) occurs: this case can be proved similarly to case (b).

\noindent (4) Region $\Omega_{4}$.
In this case, we consider (a)-(c) of Lemma 2.7, respectively.

\noindent When (a) occurs:
by using Lemma 2.3, since $\frac{1}{6}+\epsilon\leq s\leq \frac{1}{2}-2\epsilon,$ we have that
\begin{eqnarray*}
&&\left\|\Lambda^{-1}\partial_{x}(1-\partial_{x}^{2})^{-1}
\prod_{j=1}^{2}(\partial_{x}u_{j})\right\|_{Y^{s}}
\leq C\left\|\langle k\rangle^{s}\langle \sigma \rangle^{-1}\left[(|k|\mathscr{F}u_{1})*(\mathscr{F}u_{2})\right]
\right\|_{l_{k}^{2}L_{\tau}^{1}}\nonumber\\
&&\leq C\left\|\langle k\rangle ^{-2}\left[\mathscr{F}u_{1}*(\langle k\rangle ^{s}\mathscr{F}u_{2})\right]\right\|_{l_{k}^{2}L_{\tau}^{1}}\nonumber\\
&&\leq C\left\|\left[(\langle k\rangle ^{-2}\mathscr{F}u_{1})*(\langle k\rangle ^{s}\mathscr{F}u_{2})\right]\right\|_{l_{k}^{2}L_{\tau}^{1}}\nonumber\\
&&\leq C\|\langle k\rangle ^{-2}u_{1}\|_{l_{k}^{1}L_{\tau}^{1}}\|u_{2}\|_{Y^{s}}
\leq C\prod_{j=1}^{2}\|u_{j}\|_{Y^{s}}\leq C\prod_{j=1}^{2}\|u_{j}\|_{Z^{s}}.
\end{eqnarray*}
(b): $|\sigma_{1}|={\rm max}\left\{|\sigma|,|\sigma_{1}|,|\sigma_{2}|\right\}.$

\noindent We consider $|\sigma_{1}|>4{\rm max}\left\{|\sigma|,|\sigma_{2}|\right\}$ and
$|\sigma_{1}|\leq4{\rm max}\left\{|\sigma|,|\sigma_{2}|\right\}$, respectively.

\noindent If $|\sigma_{1}|>4{\rm max}\left\{|\sigma|,|\sigma_{2}|\right\}$,
 then $\supp \mathscr{F}u_{1}\subset D_{3}$ and
 $\supp \mathscr{F}u_{2}\subset D_{1}\cup D_{2}$,
 by using Lemma 2.3,
since $\frac{1}{6}+\epsilon\leq s\leq\frac{1}{2}-2\epsilon,$  we have that
\begin{eqnarray*}
&&\left\|\Lambda^{-1}\partial_{x}(1-\partial_{x}^{2})^{-1}
\prod_{j=1}^{2}(\partial_{x}u_{j})\right\|_{Y^{s}}\nonumber\\&&\leq C\left\|\Lambda^{-1}\partial_{x}(1-\partial_{x}^{2})^{-1}
\prod_{j=1}^{2}(\partial_{x}u_{j})\right\|_{Z^{s}}
\leq C\left\|\partial_{x}(1-\partial_{x}^{2})^{-1}\prod_{j=1}^{2}
(\partial_{x}u_{j})\right\|_{X_{s,-\frac{1}{6}-\epsilon}}\nonumber\\
&&\leq C\left\|\langle \sigma \rangle^{-\frac{1}{6}-\epsilon}
\left[(|k|\mathscr{F}u_{1})*(\langle k\rangle^{s}\mathscr{F}u_{2})\right]
\right\|_{l_{k}^{2}L_{\tau}^{2}}
\nonumber\\
&&\leq C\left\|(J^{1-s}\Lambda ^{s}u_{1})(J^{-s}u_{2})\right\|_{X_{0,-\frac{1}{6}-\epsilon}}
\nonumber\\&&\leq C\|u_{1}\|_{X_{1-s,s}}\|u_{2}\|_{X_{0,\frac{1}{2}}}
\leq C\prod_{j=1}^{2}\|u_{j}\|_{Z^{s}}.
\end{eqnarray*}
When $|\sigma_{1}|\leq4{\rm max}\left\{|\sigma|,|\sigma_{2}|\right\}$,
we have that $|\sigma_{1}|\sim |\sigma|$ or $|\sigma_{1}|\sim |\sigma_{2}|$.

\noindent
 When  $|\sigma_{1}|\sim |\sigma|$, this case can be proved similarly
 to case $|\sigma|={\rm max}\left\{|\sigma|,|\sigma_{1}|,|\sigma_{2}|\right\}.$

\noindent When $|\sigma_{1}|\sim |\sigma_{2}|$, we have $\supp \mathscr{F}u_{1}\subset D_{3}$
and $\supp \mathscr{F}u_{2}\subset D_{2}\bigcup D_{3}.$

\noindent When
 $\supp \mathscr{F}u_{2}\subset D_{2}$,
 by using Lemmas 2.5, 2.3, since $\frac{1}{6}+\epsilon\leq s\leq\frac{1}{2}-2\epsilon,$  we have that
\begin{eqnarray*}
&&\left\|\Lambda^{-1}\partial_{x}(1-\partial_{x}^{2})^{-1}\prod_{j=1}^{2}
(\partial_{x}u_{j})\right\|_{Y^{s}}\nonumber\\&&\leq C
\left\|\Lambda^{-1}\partial_{x}(1-\partial_{x}^{2})^{-1}
\prod_{j=1}^{2}(\partial_{x}u_{j})\right\|_{Z^{s}}
\leq C\left\|\partial_{x}(1-\partial_{x}^{2})^{-1}\prod_{j=1}^{2}
(\partial_{x}u_{j})\right\|_{X_{s,-\frac{1}{6}-\epsilon}}\nonumber\\&&
\leq C\left\|\langle k\rangle^{s}\langle \sigma \rangle^{-\frac{1}{6}-\epsilon}
\left[(|k|\mathscr{F}u_{1})*\mathscr{F}u_{2}\right]\right\|_{l_{k}^{2}L_{\tau}^{2}} \nonumber\\
&&\leq C\left\|\langle \sigma \rangle^{-\frac{1}{6}-\epsilon}\left(\langle k\rangle ^{-\frac{3}{2}+3\epsilon}\mathscr{F}u_{1}
\right)*(\langle k\rangle ^{s}
\langle \sigma\rangle ^{\frac{5}{6}-\epsilon}\mathscr{F}u_{2})
\right\|_{l_{k}^{2}l_{\tau}^{2}}\nonumber\\&&\leq C
\|u_{1}\|_{X_{1-s,s}}\|u_{2}\|_{X_{s,\frac{1}{2}}}\leq C\prod_{j=1}^{2}\|u_{j}\|_{Z^{s}}.
\end{eqnarray*}
When $\supp \mathscr{F}u_{2}\subset D_{3}$,
by using $X_{s,\frac{1}{2}+\epsilon}\hookrightarrow Y^{s}$,
since $\frac{1}{6}+\epsilon\leq s\leq\frac{1}{2}-2\epsilon$,
we have that
\begin{eqnarray*}
&&\left\|\Lambda^{-1}(1-\partial_{x}^{2})^{-1}
\prod_{j=1}^{2}(\partial_{x}u_{j})
\right\|_{Y^{s}}\leq C\left\|\langle k\rangle ^{s}\langle
\sigma\rangle ^{-\frac{1}{2}+\epsilon}
\left[|k|\mathscr{F}u_{1}*\mathscr{F}u_{2}\right]\right
\|_{l_{k}^{2}L_{\tau}^{2}}\nonumber\\
&&\leq C\left\|\langle k\rangle ^{s+\frac{1}{2}+\epsilon}\langle \sigma \rangle^{2\epsilon}
\left[(|k|\mathscr{F}u_{1})*\mathscr{F}u_{2}\right]
\right\|_{l_{k}^{\infty}L_{\tau}^{\infty}}\nonumber\\
&&\leq C\left\|\left(\langle k\rangle ^{s+\frac{3}{2}+7\epsilon}\mathscr{F}u_{1}
\right)*\mathscr{F}u_{2}\right\|_{l_{k}^{\infty}l_{\tau}^{\infty}}\leq C
\left\|\langle k\rangle ^{-3s-\frac{1}{2}+7\epsilon}\right\|_{l_{k}^{\infty}}
\prod_{j=1}^{2}\|u_{j}\|_{X_{1-s,s}}\nonumber\\&&\leq C
\prod_{j=1}^{2}\|u_{j}\|_{X_{1-s,s}}\leq C\prod_{j=1}^{2}\|u_{j}\|_{Z^{s}}.
\end{eqnarray*}
When case (c) occurs:  by using Lemma 2.5, we have that
\begin{eqnarray*}
&&\left\|\Lambda^{-1}(1-\partial_{x}^{2})^{-1}\partial_{x}
\prod_{j=1}^{2}(\partial_{x}u_{j})
\right\|_{Y^{s}}\leq C\left\|\langle k\rangle ^{s}\langle
\sigma\rangle ^{-\frac{1}{6}-\epsilon}
\left[|k|\mathscr{F}u_{1}*\mathscr{F}u_{2}\right]\right
\|_{l_{k}^{2}L_{\tau}^{2}}.
\end{eqnarray*}
By using a proof similar to case (c)  of region $\Omega_{4}$ of   Lemma 3.1, we can obtain that
\begin{eqnarray*}
&&\left\|(1-\partial_{x}^{2})^{-1}\partial_{x}
\prod_{j=1}^{2}(\partial_{x}u_{j})
\right\|_{Y^{s}}\leq C\prod_{j=1}^{2}\|u_{j}\|_{Z^{s}}.
\end{eqnarray*}
(5) In region $\Omega_{5}$.
 In this region, we consider the case $|k_{1}|\leq |k|^{-2}$
 and $|k|^{-2}\leq |k_{1}|\leq 1,$
 respectively.

 \noindent
When $|k_{1}|\leq |k|^{-2}$, by using Lemma 2.5, the Young inequality
and Cauchy-Schwartz inequality, since
$\frac{1}{6}+\epsilon\leq s\leq\frac{1}{2}-2\epsilon,$ we have that
\begin{eqnarray*}
&&\left\|\Lambda^{-1}\partial_{x}(1-\partial_{x}^{2})^{-1}\prod_{j=1}^{2}
(\partial_{x}u_{j})\right\|_{Y^{s}}\leq C\left\|\langle k\rangle^{s}\langle \sigma\rangle ^{-\frac{1}{6}-\epsilon}
\left[(|k|\mathscr{F}u_{1})*\mathscr{F}u_{2}\right]\right\|_{l_{k}^{2}L_{\tau}^{2}}\nonumber\\
&&\leq C\left\|\langle k\rangle ^{-2}\left[\mathscr{F}u_{1}*\langle k\rangle^{s}
\mathscr{F}u_{2}\right]\right\|_{l_{k}^{2}L_{\tau}^{2}}\nonumber\\
&&\leq C\left\|\mathscr{F}u_{1}\right\|_{l_{k}^{1}l_{\tau}^{2}}\|u_{2}\|_{Y^{s}}
\leq C\|\mathscr{F}u_{1}\|_{l_{k}^{2}L_{\tau}^{2}}\|u_{2}\|_{Y^{s}}
\leq C\prod_{j=1}^{2}\|u_{j}\|_{Z^{s}}.
\end{eqnarray*}
When  $|k|^{-2}\leq |k_{1}|\leq 1.$
\noindent In this  case, we consider (a)-(c) of Lemma 2.7, since
$\frac{1}{6}+\epsilon\leq s\leq\frac{1}{2}-2\epsilon,$ respectively.

\noindent
When (a) occurs:
 by using Lemma 2.7,  the Young inequality and Cauchy-Schwartz inequality, since
$\frac{1}{6}+\epsilon\leq s\leq\frac{1}{2}-2\epsilon,$
we have that
\begin{eqnarray*}
&&\left\|\Lambda^{-1}\partial_{x}(1-\partial_{x}^{2})^{-1}\prod_{j=1}^{2}
(\partial_{x}u_{j})\right\|_{Y^{s}}\leq C\left\|\langle k\rangle^{s+1}\langle \sigma\rangle ^{-1}
\left[(|k|\mathscr{F}u_{1})*\mathscr{F}u_{2}\right]\right\|_{l_{k}^{2}L_{\tau}^{1}}\nonumber\\
&&\leq C\left\|\left[\mathscr{F}u_{1}*(\langle k\rangle ^{s-1}\mathscr{F}u_{2})\right]
\right\|_{l_{k}^{2}L_{\tau}^{1}}\nonumber\\
&&\leq C\left\|\mathscr{F}u_{1}\right\|_{l_{k}^{1}l_{\tau}^{1}}\|u_{2}\|_{Y^{s}}
\leq C\|\mathscr{F}u_{1}\|_{l_{k}^{2}L_{\tau}^{2}}\|u_{2}\|_{Y^{s}}
\leq C\prod_{j=1}^{2}\|u_{j}\|_{Z^{s}}.
\end{eqnarray*}
When (b) occurs, in this case
$\supp \mathscr{F}u_{1} \subset D_{4}.$
In this case, we consider case $|\sigma_{1}|> 4{\rm max}\left\{|\sigma|, |\sigma_{2}|\right\}$
and $|\sigma_{1}|\leq4{\rm max}\left\{|\sigma|, |\sigma_{2}|\right\},$ respectively.

\noindent When consider case $|\sigma_{1}|> 4{\rm max}\left\{|\sigma|, |\sigma_{2}|\right\}$,
by using the Young inequality and Cauchy-Schwartz inequality, since
$\frac{1}{6}+\epsilon\leq s\leq\frac{1}{2}-2\epsilon,$ we have that
\begin{eqnarray*}
&&\left\|\Lambda^{-1}\partial_{x}(1-\partial_{x}^{2})^{-1}
\prod_{j=1}^{2}(\partial_{x}u_{j})\right\|_{Y^{s}}\leq
C\left\|\langle k\rangle ^{s}\langle \sigma\rangle ^{-\frac{1}{6}-\epsilon}
\left[|k|\mathscr{F}u_{1}*\mathscr{F}u_{2}\right]\right\|_{l_{k}^{2}L_{\tau}^{2}}\nonumber\\
&&\leq C\left\|\langle \sigma\rangle ^{-\frac{1}{6}-\epsilon}
\left[(\langle k\rangle^{1-s}
\langle \sigma \rangle ^{s}\mathscr{F}u_{1})*(\langle k\rangle^{-2s}\mathscr{F}u_{2})\right]
\right\|_{l_{k}^{2}L_{\tau}^{2}}\nonumber\\
&&\leq C\|u_{1}\|_{X_{1-s,s}}\|u_{2}\|_{X_{s,\frac{1}{2}}}\nonumber\\&&
\leq C\|u_{1}\|_{X_{1-s,s}}\|u_{2}\|_{X_{s,1/2}}
\leq C\prod_{j=1}^{2}\|u_{j}\|_{Z^{s}}.
\end{eqnarray*}
 When $|\sigma_{1}|\leq4{\rm max}\left\{|\sigma|, |\sigma_{2}|\right\},$
we have $|\sigma_{1}|\sim |\sigma|$ or $|\sigma_{1}|\sim |\sigma_{2}|$.

\noindent Case $|\sigma_{1}|\sim |\sigma|$ can be proved similarly to
case $|\sigma|={\rm max}\left\{|\sigma|, |\sigma_{1}|,|\sigma_{2}|\right\}.$

\noindent When $|\sigma_{1}|\sim |\sigma_{2}|$, we consider case
$\supp \mathscr{F}u_{2}\subset D_{1}$, $\supp \mathscr{F}u_{2}\subset D_{2}$
and $\supp \mathscr{F}u_{2}\subset D_{3},$ respectively.

\noindent When $\supp \mathscr{F}u_{2}\subset D_{1}$,
by using  $X_{s,\frac{1}{2}+\epsilon}\hookrightarrow Y^{s},$ the Young inequality and Cauchy-Schwartz inequality, since
$\frac{1}{6}+\epsilon\leq s\leq\frac{1}{2}-2\epsilon,$ we have that
\begin{eqnarray*}
&&\left\|\Lambda^{-1}\partial_{x}(1-\partial_{x}^{2})^{-1}
\prod_{j=1}^{2}(\partial_{x}u_{j})\right\|_{Y^{s}}\leq
C\left\|\langle k\rangle ^{s}\langle \sigma\rangle ^{-\frac{1}{2}+\epsilon}
\left[(|k|\mathscr{F}u_{1})*\mathscr{F}u_{2}\right]\right\|_{l_{k}^{2}L_{\tau}^{2}}\nonumber\\
&&\leq C\left\|\langle \sigma\rangle ^{-\frac{1}{2}+\epsilon}
\left[(\mathscr{F}u_{1})*(\langle k\rangle^{s}
\langle \sigma \rangle ^{\frac{5}{6}-\epsilon}\mathscr{F}u_{2})\right]
\right\|_{l_{k}^{2}L_{\tau}^{2}}\nonumber\\
&&\leq C\|u_{1}\|_{X_{0,\frac{1}{6}+\epsilon}}\|u_{2}\|_{X_{s,\frac{5}{6}}}\nonumber\\&&
\leq C\|u_{1}\|_{X_{s,\frac{1}{6}+\epsilon}}\|u_{2}\|_{X_{s,\frac{5}{6}}}
\leq C\prod_{j=1}^{2}\|u_{j}\|_{Z^{s}}.
\end{eqnarray*}
When $\supp \mathscr{F}u_{2}\subset D_{2}$,
by using  $X_{s,\frac{1}{2}+\epsilon}\hookrightarrow Y^{s},$ the Young inequality and Cauchy-Schwartz inequality, since
$\frac{1}{6}+\epsilon\leq s\leq\frac{1}{2}-2\epsilon,$ we have that
\begin{eqnarray*}
&&\left\|\Lambda^{-1}\partial_{x}(1-\partial_{x}^{2})^{-1}
\prod_{j=1}^{2}(\partial_{x}u_{j})\right\|_{Y^{s}}\leq
C\left\|\langle k\rangle ^{s}\langle \sigma\rangle ^{-\frac{1}{2}+\epsilon}
\left[(|k|\mathscr{F}u_{1})*\mathscr{F}u_{2}\right]\right\|_{l_{k}^{2}L_{\tau}^{2}}\nonumber\\
&&\leq C\left\|\langle \sigma\rangle ^{-\frac{1}{2}+\epsilon}
\left[(\mathscr{F}u_{1})*(\langle k\rangle^{1-s}
\langle \sigma \rangle ^{s+\frac{1}{3}-\epsilon}\mathscr{F}u_{2})\right]
\right\|_{l_{k}^{2}L_{\tau}^{2}}\nonumber\\
&&\leq C\|u_{1}\|_{X_{0,\frac{1}{6}+\epsilon}}\|u_{2}\|_{X_{1-s,s+\frac{1}{3}-\epsilon}}\nonumber\\&&
\leq C\|u_{1}\|_{X_{1-s,s}}\|u_{2}\|_{X_{1-s,s+\frac{1}{3}-\epsilon}}
\leq C\prod_{j=1}^{2}\|u_{j}\|_{Z^{s}}.
\end{eqnarray*}
When $\supp \mathscr{F}u_{2}\subset D_{3}$,
by using  $X_{s,\frac{1}{2}+\epsilon}\hookrightarrow Y^{s},$ the Young inequality and Cauchy-Schwartz inequality, since
$\frac{1}{6}+\epsilon\leq s\leq\frac{1}{2}-2\epsilon,$ we have that
\begin{eqnarray*}
&&\left\|\Lambda^{-1}\partial_{x}(1-\partial_{x}^{2})^{-1}
\prod_{j=1}^{2}(\partial_{x}u_{j})\right\|_{Y^{s}}\leq
C\left\|\langle k\rangle ^{s}\langle \sigma\rangle ^{-\frac{1}{2}+\epsilon}
\left[(|k|\mathscr{F}u_{1})*\mathscr{F}u_{2}\right]\right\|_{l_{k}^{2}L_{\tau}^{2}}\nonumber\\
&&\leq C\left\|\langle \sigma\rangle ^{-\frac{1}{2}+\epsilon}
\left[(\mathscr{F}u_{1})*(\langle k\rangle^{1-s}
\langle \sigma \rangle ^{s}\mathscr{F}u_{2})\right]
\right\|_{l_{k}^{2}L_{\tau}^{2}}\nonumber\\
&&\leq C\|u_{1}\|_{X_{0,\frac{1}{6}+\epsilon}}\|u_{2}\|_{X_{1-s,s}}\nonumber\\&&
\leq C\|u_{1}\|_{X_{1-s,s}}\|u_{2}\|_{X_{s,\frac{5}{6}}}
\leq C\prod_{j=1}^{2}\|u_{j}\|_{Z^{s}}.
\end{eqnarray*}
When (c) occurs: we consider case $|\sigma_{2}|> 4{\rm max}\left\{|\sigma|, |\sigma_{1}|\right\}$
and $|\sigma_{2}|\leq4{\rm max}\left\{|\sigma|, |\sigma_{1}|\right\},$ respectively.

\noindent When consider case $|\sigma_{2}|> 4{\rm max}\left\{|\sigma|, |\sigma_{1}|\right\}$,
by using $X_{s,\frac{1}{2}+\epsilon}\hookrightarrow Y^{s},$ the Young inequality and Cauchy-Schwartz inequality, since
$\frac{1}{6}+\epsilon\leq s\leq\frac{1}{2}-2\epsilon,$ we have that
\begin{eqnarray*}
&&\left\|\Lambda^{-1}\partial_{x}(1-\partial_{x}^{2})^{-1}
\prod_{j=1}^{2}(\partial_{x}u_{j})\right\|_{Y^{s}}\leq
C\left\|\langle k\rangle ^{s}\langle \sigma\rangle ^{-\frac{1}{2}+\epsilon}
\left[(|k|\mathscr{F}u_{1})*\mathscr{F}u_{2}\right]\right\|_{l_{k}^{2}L_{\tau}^{2}}\nonumber\\
&&\leq C\left\|\langle \sigma\rangle ^{-\frac{1}{2}+\epsilon}
\left[(\langle k\rangle^{2\epsilon-\frac{5}{3}+s}
\mathscr{F}u_{1})*(\langle k\rangle^{s}\langle \sigma \rangle ^{\frac{5}{6}-\epsilon}\mathscr{F}u_{2})\right]
\right\|_{l_{k}^{2}L_{\tau}^{2}}\nonumber\\
&&\leq C\|u_{1}\|_{X_{s,\frac{1}{6}+\epsilon}}\|u_{2}\|_{X_{s,\frac{5}{6}-\epsilon}}\nonumber\\&&
\leq C\|u_{1}\|_{X_{s,\frac{1}{6}+\epsilon}}\|u_{2}\|_{X_{s,\frac{5}{6}-\epsilon}}
\leq C\prod_{j=1}^{2}\|u_{j}\|_{Z^{s}}.
\end{eqnarray*}
When $|\sigma_{2}|\leq4{\rm max}\left\{|\sigma|, |\sigma_{1}|\right\},$
we have $|\sigma_{2}|\sim |\sigma|$ or $|\sigma_{2}|\sim |\sigma_{1}|$.

\noindent Case $|\sigma_{1}|\sim |\sigma|$ can be proved similarly to
case $|\sigma|={\rm max}\left\{|\sigma|, |\sigma_{1}|,|\sigma_{2}|\right\}.$

\noindent When $|\sigma_{2}|\sim |\sigma_{1}|$, this case can be proved similarly to
case $|\sigma_{1}|\sim |\sigma_{2}|$.

\noindent
(6)In region $\Omega_{6}$.
This case can be proved similarly to $\Omega_{4}$.

\noindent(7)In region $\Omega_{7}$.
This case can be proved similarly to $\Omega_{7}$.

\noindent (8)In region $\Omega_{8}$. In this  case, we consider (a)-(c) of Lemma 2.7, respectively.

\noindent When (a) occurs:
  by using Lemma 2.7  and  the Young inequality, since
$\frac{1}{6}+\epsilon\leq s\leq\frac{1}{2}-2\epsilon,$ we have that
\begin{eqnarray*}
&&\left\|\Lambda^{-1}\partial_{x}(1-\partial_{x}^{2})^{-1}
\prod_{j=1}^{2}(\partial_{x}u_{j})\right\|_{Y^{s}}\leq C
\left\|\langle k\rangle ^{s+1}\langle \sigma \rangle ^{-1}
\left(\mathscr{F}u_{1}*\mathscr{F}u_{2}\right) \right\|_{l_{k}^{2}L_{\tau}^{1}}
\nonumber\\&&\leq C\left\|\langle k\rangle ^{s-2}
(\mathscr{F}u_{1}*\mathscr{F}u_{2})\right\|_{l_{k}^{2}L_{\tau}^{1}}\nonumber\\&&
\leq C\left\|\langle k\rangle ^{s}\mathscr{F}u_{1}\right\|_{l_{k}^{2}L_{\tau}^{1}}
\|\langle k\rangle ^{-2}\mathscr{F}u_{2}\|_{l_{k}^{1}L_{\tau}^{1}}\leq
C\|u_{1}\|_{Y^{s}}\|u_{2}\|_{Y^{s}}\leq C\prod_{j=1}^{2}\|u_{j}\|_{Z^{s}}.
\end{eqnarray*}
When (b)  occurs:
if
 $\supp\mathscr{F}u_{1}\subset D_{2}$,   by using $X_{s,\frac{1}{2}+\epsilon}\hookrightarrow Y^{s}$ and Lemmas 2.5, 2.3, since
$\frac{1}{6}+\epsilon\leq s\leq\frac{1}{2}-2\epsilon,$ we have that
\begin{eqnarray*}
&&\left\|\Lambda^{-1}\partial_{x}(1-\partial_{x}^{2})^{-1}
\prod_{j=1}^{2}(\partial_{x}u_{j})\right\|_{Y^{s}}\nonumber\\
&&\leq C\left\|(J^{1-s}\Lambda ^{s}u_{1})(J^{s}u_{2})\right\|_{X_{0,-\frac{1}{2}+\epsilon}}
\leq C\|u_{1}\|_{X_{1-s,s}}\|u_{2}\|_{X_{s,\frac{1}{6}+\epsilon}}\leq
C\prod_{j=1}^{2}\|u_{j}\|_{Z^{s}};
\end{eqnarray*}
if
 $\supp\mathscr{F}u_{1}\subset D_{3}$,
 by using $X_{s,\frac{1}{2}+\epsilon}\hookrightarrow Y^{s}$ and Lemmas 2.5, 2.3, since
$\frac{1}{6}+\epsilon\leq s\leq\frac{1}{2}-2\epsilon,$ we have that
\begin{eqnarray*}
&&\left\|\Lambda^{-1}\partial_{x}(1-\partial_{x}^{2})^{-1}
\prod_{j=1}^{2}(\partial_{x}u_{j})\right\|_{Y^{s}}\nonumber\\
&&\leq C\left\|(J^{1-s}\Lambda ^{s}u_{1})(J^{-s}u_{2})\right\|_{X_{0,-\frac{1}{2}+\epsilon}}
\leq C\|u_{1}\|_{X_{1-s,s}}\|u_{2}\|_{X_{s,\frac{1}{6}+\epsilon}}\leq
C\prod_{j=1}^{2}\|u_{j}\|_{Z^{s}}.
\end{eqnarray*}
 Case (c)  can be proved similarly to Case (b).

We have completed the proof of Lemma 3.2.

\noindent {\bf Remark 4.}  In the process of proving Lemma 3.2, the cases $ \Omega_{3}$ is  the most difficult to deal with and cases $ \Omega_{j}$ with $j=2,3$
require restriction $s>\frac{1}{6},$  more precisely,  case (b) of region $\Omega_{2}$ determines that $s\geq \frac{1}{6}+\epsilon$ is necessary.

\begin{Lemma}\label{Lemma3.3}
Let  $\frac{1}{6}+\epsilon\leq s\leq\frac{1}{2}-2\epsilon$,
$0<\epsilon\ll1$. Then, we have that
\begin{eqnarray}
      \left\|\Lambda^{-1}\partial_{x}(1-\partial_{x}^{2})^{-1}
      \prod_{j=1}^{2}(\partial_{x}u_{j})
      \right\|_{Z^{s}}
      \leq C\prod\limits_{j=1}^{2}\|u_{j}\|_{Z^{s}}.
        \label{3.03}
\end{eqnarray}
\end{Lemma}
{\bf Proof.} Combining the definition of $Z^{s}$ with Lemmas 3.1, 3.2, we have Lemma 3.3.

We have completed the proof of Lemma 3.3.

By using  a proof similar to Lemma 3.3,  we have Lemmas 3.4, 3.5.

\begin{Lemma}\label{Lemma3.4}
Let  $\frac{1}{6}+\epsilon\leq s\leq\frac{1}{2}-2\epsilon$,
$0<\epsilon\ll1$. Then, we have that
\begin{eqnarray}
      \left\|\Lambda^{-1}\partial_{x}\prod_{j=1}^{2}(u_{j})
      \right\|_{Z^{s}}
      \leq C\prod\limits_{j=1}^{2}\|u_{j}\|_{Z^{s}}.
        \label{3.04}
\end{eqnarray}
\end{Lemma}

\begin{Lemma}\label{Lemma3.5}
Let  $\frac{1}{6}+\epsilon\leq s\leq\frac{1}{2}-2\epsilon$,
$0<\epsilon\ll1$. Then, we have
\begin{eqnarray}
      \left\|\Lambda^{-1}\partial_{x}(1-\partial_{x}^{2})^{-1}\prod_{j=1}^{2}(u_{j})
      \right\|_{Z^{s}}
      \leq C\prod\limits_{j=1}^{2}\|u_{j}\|_{Z^{s}}.
        \label{3.05}
\end{eqnarray}
\end{Lemma}

\bigskip
\bigskip

\noindent {\large\bf 4. Proof of Theorem  1.1}

\setcounter{equation}{0}

 \setcounter{Theorem}{0}

\setcounter{Lemma}{0}

\setcounter{section}{4}
In this section,  we prove Theorem 1.1.
\begin{proof}Let $N\gg1$, $a\in \dot{Z}$ and
\begin{eqnarray*}
&&\mathscr{F}u_{1}(k,\tau)=\left(\chi_{(N)}(k)+\chi_{(N)}(-k)\right)\chi_{[-1,1]}(\tau-k^{3}),\nonumber\\&& \mathscr{F}u_{2}(k,\tau)=\left(\chi_{(1-N)}(k)+\chi_{(1-N)}(-k)\right)\chi_{[-1,1]}(\tau-k^{3}),
\end{eqnarray*}
here
\begin{eqnarray*}
\chi_{a}(k)=1 \quad {\rm if}\quad k=a, \chi_{a}(k)=0\quad {\rm if}\quad k\neq a,
\end{eqnarray*}
and
\begin{eqnarray*}
\chi_{[-1,1]}(\sigma)=1\quad {\rm if} \quad |\sigma|\leq 1, \chi_{[-1,1]}=0,\quad {\rm if}\quad |\sigma|>1.
\end{eqnarray*}
Then, we have that
\begin{eqnarray*}
\|u_{j}\|_{W^{s}}\sim N^{s},j=1,2.
\end{eqnarray*}
Let
\begin{eqnarray*}
&&R_{1}(k_{1},k_{2})=\chi_{N}(k_{1})\chi_{(1-N)}(k_{2}), R_{2}(k_{1},k_{2})=\chi_{N}(k_{1})\chi_{(1-N)}(-k_{2}), \nonumber\\&&
R_{3}(k_{1},k_{2})=\chi_{N}(-k_{1})\chi_{(1-N)}(k_{2}), R_{4}(k_{1},k_{2})=\chi_{N}(-k_{1})\chi_{(1-N)}(-k_{2}).
\end{eqnarray*}
Then, by using a direct computation, we have that
\begin{eqnarray*}
&&\left\|\mathscr{F}^{-1}\left[\langle \tau-k^{3}\rangle^{-1}\mathscr{F}F(u_{1},u_{2})\right]\right\|_{W^{s}}\nonumber\\&&\hspace{-1cm}
=\left\|\sum_{j=1}^{4}\int_{\dot{Z}}\frac{|k|^{s+1}}{1+k^{2}}
\left[k^{2}+3+k_{1}k_{2}\right]R_{j}(k_{1},k_{2})\left(\int_{\SR}\langle\sigma\rangle^{-1/2}
\chi_{[-1,1]}(\sigma_{1})\chi_{[-1,1]}(\sigma_{2})d\sigma_{1}\right)dn_{1}\right\|_{L_{n\sigma}^{2}}.
\end{eqnarray*}
From Lemma 2.7, we have that
\begin{eqnarray*}
\langle\sigma\rangle \sim |kk_{1}k_{2}|,
\end{eqnarray*}
since $|\sigma_{j}|\leq 1$ with $j=1,2.$ Thus, we have that
\begin{eqnarray*}
\int_{\SR}\langle \sigma\rangle^{-1/2}\chi_{[-1,1]}(\sigma_{1})\chi_{[-1,1]}(\sigma_{2})d\sigma_{1}\geq C|kk_{1}k_{2}|^{-1/2}.
\end{eqnarray*}
By using a direct computation, we have that
\begin{eqnarray*}
&&\left\|\mathscr{F}^{-1}\left[\langle \tau-k^{3}\rangle^{-1}\mathscr{F}F(u_{1},u_{2})\right]\right\|_{W^{s}}
\nonumber\\&&\geq C\left\|\sum_{j=1}^{4}\int_{\dot{Z}}\frac{|k|^{s+1}}{1+k^{2}}
\left[k^{2}+3+k_{1}k_{2}\right]R_{j}(k_{1},k_{2})|kk_{1}k_{2}|^{-1/2}dk_{1}\right\|_{L_{k}^{2}}\geq CN.
\end{eqnarray*}
If (\ref{1.06}) is invalid, then we have
\begin{eqnarray}
&&CN\leq \left\|\mathscr{F}^{-1}\left[\langle \tau-k^{3}\rangle^{-1}\mathscr{F}F(u_{1},u_{2})\right]\right\|_{X_{s,\frac{1}{2}}}\nonumber\\&&\leq C
\left\|\mathscr{F}^{-1}\left[\langle \tau-k^{3}\rangle^{-1}\mathscr{F}F(u_{1},u_{2})\right]\right\|_{W^{s}}\leq C\prod_{j=1}^{2}\|u_{j}\|_{W^{s}}\sim N^{2s}.\label{4.01}
\end{eqnarray}
We obtain the contradiction since $s<\frac{1}{2}.$
\end{proof}

We have completed the proof of Theorem 1.1.

\noindent {\large\bf 5. Proof of Theorem  1.2}

\setcounter{equation}{0}

 \setcounter{Theorem}{0}

\setcounter{Lemma}{0}

\setcounter{section}{4}
Now we are in a position to prove Theorem 1.2.
Let
\begin{eqnarray*}
F(t)=\frac{1}{2}\partial_{x}(u^{2})
+\partial_{x}(1-\partial_{x}^{2})^{-1}\left[u^{2}+\frac{1}{2}u_{x}^{2}\right].
\end{eqnarray*}
We define
\begin{eqnarray*}
&&\Phi(u)=\eta(t) S(t)\phi-\eta(t) \int_{0}^{t}S(t-t^{'})F(t^{'})dt^{'},\label{4.01}\nonumber\\
&&B=\left\{u\in Z^{s}: \quad \|u\|_{ Z^{s}}\leq 2C\|\phi\|_{H^{s}(\mathbf{T})}\right\}.
\end{eqnarray*}
By using   Lemmas 2.3-2.4 and Lemma 3.3-3.5,  we have that
\begin{eqnarray*}
&&\left\|\Phi(u)\right\|_{Z^{s}}\leq \left\| \eta(t)S(t)\phi\right\|_{Z^{s}}
+\left\|\eta(t) \int_{0}^{t}S(t-t^{'})
F(t^{'})dt^{'}\right\|_{Z^{s}}\nonumber\\&&\leq C\|\phi\|_{H^{s}(\mathbf{T})}
+C\left\|\Lambda^{-1}F(t)\right\|_{Z^{s}}\nonumber\\
&&\leq C\|\phi\|_{H^{s}(\mathbf{T})}+C\|u\|_{Z^{s}}^{2}.
\end{eqnarray*}
Obviously, for sufficiently small $\|\phi\|_{H^{s}(\mathbf{T})}$, we have that
\begin{eqnarray*}
\left\|\Phi(u)\right\|_{Z^{s}}\leq C\|\phi\|_{H^{s}(\mathbf{T})}+4C^{3}\|\phi\|_{H^{s}(\mathbf{T})}^{2}\leq 2C\|\phi\|_{H^{s}(\mathbf{T})}
\end{eqnarray*}
on closed ball $B$.
For $u,v\in B$,    when $\|\phi\|_{H^{s}(0,2\pi)}$ is sufficiently small,  we have that
\begin{eqnarray*}
&&\left\|\Phi(u)-\Phi(v)\right\|_{Z^{s}}\nonumber\\&&\leq C
\left(\|u\|_{Z^{s}}+\|v\|_{Z^{s}}\right)\|u-v\|_{Z^{s}}\leq 2C\|\phi\|_{H^{s}(\mathbf{T})}\|u-v\|_{Z^{s}}\leq \frac{1}{2}\|u-v\|_{Z^{s}}.
\end{eqnarray*}

The proof of the rest of Theorem 1.2 is standard, which can be found in \cite{MT,TKato}, thus, we omit the process.

\bigskip
\bigskip

\leftline{\large \bf Acknowledgments}

\bigskip

\noindent

 This work is supported by the Natural Science Foundation of China
 under grant numbers 11171116, 11471330 and 11401180. The first author is also
 supported in part by the Fundamental Research Funds for the
 Central Universities of China under the grant number 2012ZZ0072.
 The second author is  supported by the
 NSF of China (No.11371367) and Fundamental
 research program of NUDT(JC12-02-03).


\bigskip

  \bigskip


\leftline{\large\bf  References}



\begin{thebibliography}{99}

\bibitem{BSS}
R. Beals, D. Sattinger, J. Szmigielski, Multipeakons and a theorem of Stieltjes, {\it Inverse Problems} 15(1999), 1-4.

 \bibitem{BT}
 I.  Bejenaru, T. Tao,  Sharp well-posedness and ill-posedness results for a
quadratic non-linear Schr\"odinger equation, {\it J. Funct. Anal.} 233(2006), 228-259.

 \bibitem{B}
 J. Bourgain,
 Fourier transform restriction phenomena for certain lattice subsets and
 applications to nonlinear evolution equations,  part I: Schr\"odinger equations,
  {\it Geom. Funct. Anal.} 3(1993),  107-156.


\bibitem{Bourgain93}
J. Bourgain, Fourier transform restriction phenomena for certain lattice subsets
and applications to nonlinear evolution equations,
part II: The KdV equation,  {\it Geom.   Funct. Anal.} 3(1993), 209-262.

\bibitem{Bourgain97}
 J. Bourgain, Periodic Korteweg  de vries equation with measures as initial data,
 {\it  Sel. Math.} 3(1997), 115-159.

 \bibitem{B98}
 J. Bourgain,  Refinements of Strichartz' inequality and applications to 2 D-NLS with critical nonlinearity,
 {\it  Int. Math. Res. Not.}
  5(1998),   253-283.

 \bibitem{BCARMA}
 A. Bressan, A. Constantin, Global  conservative solutions of the Camassa-Holm equation, {\it Arch. Ration. Mech. Anal.} 183(2007), 215-239.

 \bibitem{BCAA} A. Bressan, A. Constantin, Global dissipative solutions of the Camassa-Holm equation, {\it Anal. Appl.} 5(2007), 1-27.

\bibitem{BDIE} P. Byers, The Cauchy problem for a fifth order evolution equation, {\it Diff. Int. Eqns.} 16(2003), 537-556.

\bibitem{B2003} P. Byers, The initial value problem for a
KdV-type equation and a related bilinear estimate, dissertation, University of Notre
Dame, 2003.


\bibitem{CH} R. Camassa, D. Holm, An integrable shallow water equation with peaked solutions,
 {\it Phys. Rev. Lett.}  71(1993), 1661-1664.
 \bibitem{CHH} R. Camassa, D. Holm, J. Hyman, A new integrable shallow water equation, {\it Adv. Appl. Mech.} 31(1994), 1-33.

\bibitem{CKSTT}  J. Colliander, M. Keel, G. Staffilani,  H. Takaoka,  T. Tao,
Sharp global well-posedness for KdV and modified KdV on $\R$  and $\mathbf{T}$,
{\it J. Amer. Math. Soc.} 16(2003),  705-749.

\bibitem{C} A. Constantin,  The Hamiltonian structure of the Camassa-Holm equation, {\it  Exposition. Math.} 15(1997),  53-85.
\bibitem{Con2000} A. Constantin, Existence of permanent and breaking waves for a shallow water equation: a geometric approach,
{\it Ann. Inst. Fourier (Grenoble)} 50(2000), 321-362.

\bibitem{C2000}A. Constantin, On the blow-up of solutions of a periodic shallow water equation, {\it  J. Nonlinear Sci.} 10(2000),  391-399.

\bibitem{C2001}A.  Constantin,  On the scattering problem for the Camassa-Holm equation, {\it Proc. R. Soc. Lond. Proc. Ser. A } 457(2001), 953-970.

\bibitem{C2006} A. Constantin, The trajectories of particles in Stokes waves, {\it Invent. Math.} 166(2006), 523-535.



\bibitem{CE} A. Constantin, J. Escher,  Wave breaking for nonlinear nonlocal shallow water equations, {\it  Acta Math.} 181(1998),  229-243.

\bibitem{CECPAM}A. Constantin, J. Escher,  Well-posedness, global existence, and
blowup phenomena for a periodic quasi-linear hyperbolic equation, {\it  Comm. Pure Appl. Math.} 51(1998),  475-504.


\bibitem{CE1998}A.  Constantin, J. Escher, Global weak solutions for a shallow water equation,
{\it  Indiana Univ. Math. J.} 47(1998), 1527-1545.

\bibitem{CEA}A.  Constantin, J. Escher,  Global existence and blow-up for a shallow water equation,
 {\it Ann. Scuola Norm. Sup. Pisa Cl. Sci.} 26(1998),  303-328.



\bibitem{CE2000}A.  Constantin, J. Escher,  On the blow-up rate and the blow-up set of breaking waves for a shallow water equation,
 {\it Math. Z.} 233(2000),  75-91.

\bibitem{CEB2007} A. Constantin, J. Escher, Particle trajectories in solitary water waves, {\it Bull. AMer. Math. Soc.} 44(2007), 423-431.



\bibitem{CKL}A. Constantin,  B. Kolev, J.  Lenells,  Integrability of invariant metrics on the Virasoro group,
 {\it  Phys. Lett. A} 350(2006),  75-80.

\bibitem{CK}A. Constantin, B. Kolev, Geodesic flow on the diffeomorphism group of
the circle, {\it  Comment. Math. Helv.} 78(2003),  787-804.

 \bibitem{CLa} A. Constantin, D. Lannes, The hydrodynamical relevance of the Camssa-Holm and Degasperis-Procesi equation,
{\it Arch. Ration. Mech. Anal.} 192(2009), 165-186.




 \bibitem{CJ} A. Constantin, R. S. Johnson, Propagation of very long water waves, with vorticity,
 over variable depth, with applications to tsunamis, {\it Fluid Dynam. Res.}40 (2008), 175-211.

 \bibitem{CMCPAM} A. Constantin, H. P. Mckean, A shallow water
on the circle,
{\it Comm. Pure Appl. Math.} 52(1999), 949-982.


 \bibitem{CM}A.  Constantin, L.  Molinet,  Global weak solutions for a shallow water equation, {\it  Comm. Math. Phys.} 211(2000), 45-61.




 \bibitem{CMP}A. Constantin, L.  Molinet,
  Orbital stability of solitary waves for a shallow water equation, {\it  Phys. D} 157(2001),  75-89.

\bibitem {CS} A. Constantin, W.  Strauss,  Stability of peakons, {\it  Comm. Pure Appl. Math.} 53(2000),  603-610.
\bibitem{CSPLA} A. Constantin, W. A. Strauss, Stability of a class of solitary waves in compressible elastic rods, {\it Phys. Lett. A} 270(2000), 140-148.


\bibitem{CSNS} A. Constantin, W. A. Strauss,  Stability of the Camassa-Holm solitons, {\it  J. Nonlinear Sci.} 12(2002),  415-422.

\bibitem{DH} H. H. Dai, Model equations for nonlinear dispersive waves in a compressible
Mooney-Rivlin rod, {\it Acta Mech.} 127(1998), 193-207.



 \bibitem{D2003} R. Danchin,  A note on well-posedness for Camassa-Holm equation, {\it  J. Diff. Eqns.} 192(2003),  429-444.

\bibitem{DGH} H. R. Dullin, G. A. Gottwald, D. D. Holm, An integrable
shallow water equation with linear and nonlinear dispersion,
{\it Phys. Rev. Lett.} 87(2001), 4501-4504 .




\bibitem{EY}J. Escher, Z. Y. Yin, Initial boundary value problems of the
Camassa-Holm equation, {\it  Comm. Partial Differential
Equations} 33(2008),  377-395.

\bibitem{EYJFA} J. Escher, Z. Y. Yin, Initial boundary value problems
for nonlinear dispersive wave equations, {\it J. Funct. Anal.} 256(2009), 479-508.


\bibitem{FF}A. Fokas, B. Fuchssteiner, Symplectic structures, their
$B\ddot{a}$klund transformations and hereditary
symmetries, {\it  Phys. D.}  71(1981),  47-66.


\bibitem{Go} J. Gorsky, On the Cauchy problem for a KdV-type equation
 on the circle, dissertation, University of Notre Dame, 2004.





\bibitem{G} Z. H.  Guo,  Global well-posedness of Korteweg-de Vries
equation in $H^{-3/4}(\R),$
{\it Journal de Math$\acute{e}$matiques Pures et Appliqu$\acute{e}$es}, 91(2009), 583-597.




\bibitem{HM1998} A. A. Himonas, G. Misiolek, The  Cauchy   problem for
a shallow water type equation, {\it  Comm. Partial   Diff. Eqns.}
23(1998), 123-139.




\bibitem{HM2000} A. A. Himonas, G. Misiolek,  Well-posedness of the Cauchy
problem for a shallow water equation on the circle,
{\it J. Diff. Eqns.} 161(2000), 479-495.

\bibitem{HM}A. A. Himonas, G. Misiolek, The initial value problem for a
fifth order shallow water in: Analysis, Geometry, Number Theory:
The Mathematics of Leon Ehrenpreis, in: Contemp. Math. Vol. 251, Amer.
Math. Soc., Providence, RI, 2000, pp. 309-320.

\bibitem{HMPZ} A. A. Himonas,G.  Misiolek, G.  Ponce, Y. Zhou,  Persistence properties and unique continuation of solutions of the Camassa-Holm equation,
 {\it Comm. Math. Phys.} 271(2007), 511-522.
\bibitem{HM2005} A. A. Himonas, G. Misiolek,  High-frequency smooth solutions and well-posedness of the Camassa-Holm equation,{\it Int. Math. Res. Not.} 51(2005), 3135-3151.


\bibitem{IK}
A. D. Ionescu, C. E. Kenig, Global well-posedness of the Benjamin-Ono equation
in low-regularity spaces, {\it J. Amer. Math. Soc.} 20(2007), 753-798.

\bibitem{IKT}
A. D. Ionescu, C. E. Kenig, D. Tataru,  Global well-posedness of the KP-I
initial-value problem in the energy space, {\it Invent. Math.} 173(2008), 265-304.

\bibitem{Iv2006}R. I. Ivanov, Exended Camassa-Holm hierarchy and conserved quantities, {\it Z. Naturforsch} 61(2006), 133-138.
\bibitem{Iv2007} R. I. Ivanov, Water waves and integrability, {\it Philos. Trans. R. Soc. Lond. Ser. A.} 365(2007), 2267-2280.
\bibitem{J} R. S. Johnson, Camassa-Holm, Korteweg-de Vries and related models for water waves, {\it J. Fluid Mech.} 457(2002), 63-82.

\bibitem{KL} H. Kalisch, J.  Lenells,  Numerical study of traveling-wave solutions for the Camassa-Holm equation,
 {\it  Chaos Solitons Fractals} 25(2005),  287-298.



\bibitem{KT2006} T. Kappeler and P. Topalov, Global wellposedness of KdV in
$H^{-1}({\mathbf T},{\mathbf R})$, {\it Duke Math. J.}  135(2006),  327-360.

\bibitem{TKato} T. K. Kato, Local well-posedness for Kawahara equation,
{Adv. Diff. Equ.} 16(2011), 257-287.

\bibitem{Kato}
T. Kato, Low regularity well-posedness for the periodic Kawahara  equation,
{\it Diff. Int. Eqns.} 25(2012),  1011-1036.




\bibitem{KPV1996}
C. E. Kenig,  G. Ponce, L. Vega, A bilinear estimate with applications to the KdV equation,
 {\it J. Amer. Math. Soc.} 9(1996), 573-603.


\bibitem{KPV2001}
C. E. Kenig,  G. Ponce, L. Vega, On the ill-posedness of some canonical dispersive equations,
{\it Duke Math. J.} 106(2001), 617-633.




\bibitem{Kis}
 N. Kishimoto, Well-posedness of the Cauchy problem for the Korteweg-de
Vries  equation at the critical regularity, {\it Diff. Int. Eqns.}  22(2009), 447-464.

\bibitem{KT} N. Kishimoto, K. Tsugawa, Local well-posedness for
quadratic Schr\"odinger equation
and ``good" Boussinesq equation, {\it Diff. Int. Eqns.} 23(2010), 463-493.


\bibitem{Ko}B. Kolev, Bi-Hamiltonian systems on the dual of the Lie algebra of vector fields of the circle and periodic
shallow  water equations, {\it Philos. Trans. R. Soc. Lond. Ser. A} 365(2007), 2333-2357.


\bibitem{Lak} M. Lakshmanan, Integrable nonlinear wave equations and possible
 connections to tsunami dynamics, in: Tsunami and Nonlinear waves, Spinger, Berlin, 2007, 31-49.


\bibitem{L2004IMRN} J. Lenells,   The correspondence between KdV and Camassa-Holm,
 {\it  Int. Math. Res. Not.} 71(2004),   3797-3811.










\bibitem{LIMRN} J. Lenells,  Stability of periodic peakons, {\it Int. Math. Res. Not.} 10(2004),  485-499.

\bibitem{L2005}J.  Lenells, Stability for the periodic Camassa-Holm equation, {\it  Math. Scand.} 97(2005),  188-200.

\bibitem{LPA}J. Lenells,
 Conservation laws of the Camassa-Holm equation, {\it J. Phys. A} 38(2005), 869-880.



\bibitem{L2005JDE} J.  Lenells,  Traveling wave solutions of the Camassa-Holm equation,
 {\it  J. Differential Equations} 217(2005), 393-430.






\bibitem{LYLH} S. M. Li, W. Yan, Y. S. Li, J. H. Huang, The Cauchy problem for
a higher order shallow water  type equation on the circle,
{\it J. Diff. Eqns.} 259(2015), 4863-4896.

\bibitem{LYY}Y. S. Li,  W. Yan,   X. Y.  Yang,
Well-posedness of a higher order modified Camassa-Holm equation
 in spaces of low regularity,
{\it J. Evol. Eqns.} 10(2010),  465-486.



\bibitem{L} B. P. Liu, A priori bounds for KdV equation below $H^{-3/4}$,
{\it J. Funct. Anal.,} 268(2015), 501-554.


\bibitem{LJ} X. Liu, Y. Jin, The Cauchy problem of a shallow water equation,
{\it Acta Math. Sin.(Engl. Ser.)} 30(2004), 1-16.

\bibitem{LO} Y. Li, P. Olver, Well-posedness and blow-up solutions for an integrable
nonlinearly dispersive model wave equation, {\it J. Diff. Eqns.} 162(2000), 27-63.


\bibitem{LY} Y. S. Li, X. Y. Yang,  Global well-posedness for a fifth-order
shallow water equation on the circle,
 {\it Acta Math. Sci. Ser. B Engl. Ed.} 31(2011),  1303-1317.


\bibitem{Molinet} L. Molinet,  Sharp ill-posedness results for the KdV and mKdV
equations on the torus, {Advances in Mathematics,}  230(2012), 1895-1930.

\bibitem{MT} T. Muramatu, S. Taoka,
The initial data value problem for the 1-D semilinear Schr$\ddot{o}$dinger
equation in the Besov space,
{\it J. Math. Soc. Japan.} 56(2004), 853-888.





\bibitem{NTT} K. Nakanishi, H. Takaoka, Y. Tsutsumi,
Local well-posedness in low regularity of the mKdV equation with
periodic boundary  condition,
 {\it  Discrete Contin. Dyn. Syst.}, 28(2010), 1635-1654.


\bibitem{O} E. A. Olson,  Well posedness for a higher order modified Camassa-Holm
equation, {\it J. Diff. Eqns.},  246(2009), 4154-4172.

\bibitem{R} G. Rodriguez-Blanco, On the Cauchy problem for the Camassa-Holm
equation, {\it Nonlinear Anal.} 46(2001), 309-327.


\bibitem{TT} H. Takoaoka, Y. Tsutsumi, Well-posedness of the Cauchy problem
for the modified KdV equation with periodic boundary condition, {\it Int. Math. Res. Not.}56(2004), 3009-3040.

\bibitem{T}  T. Tao, Multilinear weighted convolution of $L^2$ functions, and
applicationsto non-linear dispersive equations, {\it Amer. J. Math.} 123(2001), 839-908.

\bibitem{To}
J. F. Toland, Stokes waves, {\it Topol. Methods Nonlinear Anal.} 7(1996), 1-48.

\bibitem{WC}
H. Wang, S. B. Cui, Global well-posedness of the Cauchy problem of the fifth order
shallow water equation, {\it J. Diff. Eqns.} 230(2006), 600-613.





 \bibitem{XZCPAM} Z. P. Xin, P. Zhang, On the weak solutions to a shallow water equation,
 {\it Comm. Pure Appl. Math.} 53(2000), 1411-1433.

\bibitem{XZCPDE} Z. P. Xin, P. Zhang,  On the uniqueness and
large time behavior of the weak solutions to a shallow water equation.
{\it Comm. Partial Diff. Eqns.} 27(2002),  1815-1844.


\bibitem{YL}X. Y. Yang, Y. S.  Li, Global well-posedness for a fifth-order
shallow water equation in Sobolev spaces,
 {\it J. Diff. Eqns.} 248(2010),  1458-1472.

\bibitem{Y} Z. Yin, Well-posedness, global existence and blowup phenomena for an
integrable shallow water equation, {\it Discrete
Contin. Dyn. Syst. Ser. A} 10(2004), 393-411.



\end{thebibliography}
\end{document}